\documentclass[final]{siamltex}

\usepackage{amsfonts}
\usepackage{latexsym}
\usepackage[notcite,notref]{showkeys}
\usepackage{amsmath, amssymb}
\usepackage{algorithm}
\usepackage{algpseudocode}
\usepackage{graphicx}
\usepackage[colorlinks,linkcolor=blue]{hyperref}

\newtheorem{thm}{Theorem}
\newtheorem{lem}[thm]{Lemma}
\newtheorem{cor}[thm]{Corollary}
\newtheorem{prop}[thm]{Proposition}
\newtheorem{remark}[thm]{Remark}
\newtheorem{assumption}{Assumption}

\newcommand{\tsum}{\textstyle\sum}

\newcommand{\B}{{\mathcal{B}}}
\newcommand{\N}{{\mathcal{N}}}

\newcommand{\beq}{\begin{equation}}
\newcommand{\eeq}{\end{equation}}
\newcommand{\beqa}{\begin{eqnarray}}
\newcommand{\eeqa}{\end{eqnarray}}
\newcommand{\beqas}{\begin{eqnarray*}}
\newcommand{\eeqas}{\end{eqnarray*}}
\newcommand{\bi}{\begin{itemize}}
\newcommand{\ei}{\end{itemize}}
\newcommand{\ba}{\begin{array}}
\newcommand{\ea}{\end{array}}

\def\argmin{{\rm argmin}}

\def\Argmin{{\rm Argmin}}

\def\exp{{\rm exp}}

\def\vgap{\vspace*{.1in}}

\def\setU{{X}}

\newcommand{\bbe}{\mathbb{E}}

\def\Prob{{\hbox{\rm Prob}}}
\newcommand{\bbr}{\mathbb{R}}

\title{
Algorithms for stochastic optimization \\
with function or expectation constraints
\thanks{Part of the results were first presented at the Annual INFORMS meeting
in Oct, 2015, \url{https://informs.emeetingsonline.com/emeetings/formbuilder/clustersessiondtl.asp?csnno=24236&mmnno=272&ppnno=91687} and
summarized in a previous version entitled ``Algorithms for stochastic optimization with expectation constraints" in 2016.
Please note that this version contains the corrections to a few problems in the paper published by {\em Computational Optimization and Applications}.}}

\author{
     Guanghui Lan
    \thanks{H. Milton Stewart School of Industrial and Systems
    Engineering, Georgia Institute of Technology, Atlanta, GA, 30332.
    (email: {\tt george.lan@isye.gatech.edu}). This author has been supported by NSF CMMI 1637474}
    \and
  Zhiqiang Zhou
    \thanks{H. Milton Stewart School of Industrial and Systems
    Engineering, Georgia Institute of Technology, Atlanta, GA, 30332.
    (email: {\tt zzhoubrian@gatech.edu}). }
}

\begin{document}

\maketitle

\begin{abstract}
This paper considers the problem of minimizing an expectation function over a closed convex set, coupled with a {\color{black} function or expectation} constraint on either
decision variables or problem parameters. We first present a new stochastic approximation (SA) type algorithm, namely the cooperative SA (CSA),
to handle problems with the constraint on devision variables. We show that this algorithm exhibits
the optimal ${\cal O}(1/\epsilon^2)$ rate of convergence, in terms of both optimality gap and constraint violation,
when the objective and constraint functions are generally convex, where $\epsilon$ denotes the optimality gap and infeasibility. Moreover, we show that
this rate of convergence can be improved to
${\cal O}(1/\epsilon)$ if the objective and constraint functions are strongly convex.
We then present a variant of CSA, namely the cooperative stochastic parameter approximation (CSPA) algorithm, to deal with
the situation when
the constraint is defined over
problem parameters and show that it exhibits similar optimal rate of convergence to CSA.
It is worth noting that CSA and CSPA are primal methods  which do not require the iterations on the dual space and/or
the estimation on the size of the dual variables.
To the best of our knowledge, this is the first time that such optimal SA methods for
solving function or expectation constrained stochastic optimization are presented in the literature.


\vspace{.1in}


\vspace{.07in}


\end{abstract}

\vspace{0.1cm}

\setcounter{equation}{0}
\section{Introduction}
In this paper, we study two related stochastic programming (SP) problems with function or expectation constraints.
The first one is a classical SP problem with the {\color{black} function constraint} over the decision variables, formally defined as
\begin{equation}\label{1.1}
\begin{aligned}
\min \  & f(x):=\bbe[F(x,\xi)]\\
  {\rm s.t.}\  &  g(x)\leq 0, \\
   & x\in X,
\end{aligned}
\end{equation}
where $X \subseteq \bbr^n$ is a convex compact set, $\xi$ are
random vectors supported on ${\cal P} \subseteq \bbr^p$,
$F(x, \xi):X\times {\cal P}\mapsto \bbr$ and  {\color{black} $g(x):X\mapsto \bbr$ are closed convex functions w.r.t. $x$
for  a.e. $\xi \in {\cal P}$}. Moreover, we assume that $\xi$ are independent of $x$.
Under these assumptions, \eqref{1.1} is a convex optimization problem.

{\color{black} In particular, the constraint function $g(x)$ in problem \eqref{1.1} can be given in the form of expectation as
\begin{equation}\label{1.1exp}
g(x) := \bbe_{\xi}[G(x,\xi)],
\end{equation}
where $G(x, \xi):X\times {\color{black} \cal P}\mapsto \bbr$ are closed convex functions w.r.t. $x$
for  a.e. $\xi \in {\color{black} \cal P}$.
Such problems have} many applications in operations research, finance and data analysis.
One motivating example is SP with the conditional value at risk (CVaR) constraint. In an important work \cite{roc00}, Rockafellar and  Uryasev shows that
a class of asset allocation problem can be modeled as
\begin{equation} \label{CVaR}
\begin{array}{ll}
  \min_{x, \tau} &  -\mu^T x  \\
  \mbox{s.t.} & \tau + \tfrac{1}{\beta}\mathbb{E}\{[-\xi^T x-\tau]_+\}\leq 0, \\
   &  \tsum_{i=1}^n x_i=1, x \ge 0,
\end{array}
\end{equation}
where $\xi$ denotes the random return with mean $\mu = \mathbb{E}[\xi]$. Expectation constraints also play an important
role in providing tight convex approximation to chance constrained problems (e.g., Nemirovksi and Shapiro \cite{nemirovski2006convex}).
Some other important applications of \eqref{1.1} can be found in semi-supervised learning (see, e.g., \cite{ChaSchZien06}).
For example, one can use the objective function to define the fidelity of the model for the labelled data, while using
the constraint to enforce some other properties of the model for the unlabelled data (e.g., proximity for data with similar features).

While problem \eqref{1.1} covers a wide class of problems with constraints over the decision variables,
in practice we often encounter the situation where the constraint is defined over
the problem parameters. Under these circumstances our goal is
 to find a pair of parameters $x^*$ and decision variables $y^*(x^*)$ such that 
\begin{align}
 y^*(x^*) &\in \Argmin_{y\in Y} \left\{ \phi(x^*, y):=\bbe[\Phi(x^*,y,\zeta)]\right\}, \label{para1}\\
 x^* &\in \left\{ x \in X | g(x):= \bbe[G(x,\xi)] \leq 0 \right\}. \label{para2}
\end{align}
Here $\Phi(x,y,\zeta)$ is convex w.r.t. $y$ for a.e. $\zeta \in {\color{black} {\cal Q} \subseteq \bbr^q}$ but possibly nonconvex w.r.t. $(x,y)$ jointly, and
$g(\cdot)$ is convex w.r.t. $x$.
Moreover, we assume that $\zeta$ is independent of $x$ and $y$, while
$\zeta$ is not necessarily independent of $x^*$. Note that \eqref{para1}-\eqref{para2}
defines a pair of optimization and feasibility problems coupled through the following ways: a)
the solution to \eqref{para2} defines an admissible parameter of \eqref{para1};  b)
$\xi$
can be a random variable with probability distribution parameterized
by $x^*$.

Problem \eqref{para1}-\eqref{para2} also has many applications, especially in data analysis.
One such example is to learn a classifier $w$
with a certain metric $\bar A$ using the support vector machine model:
\begin{align}
&\min_w \mathbb{E}[l(w;(\bar A^{\tfrac{1}{2}}u,v))] + \tfrac{\lambda}{2}\|w\|^2, \label{svm} \\
&\bar A \in \left\{A\succeq 0 |  \mathbb{E}[|{\rm Tr}(A(u_i-v_j)(u_i-v_j)^T)-b_{ij}|]\leq 0, {\rm Tr}(A) \leq C\right\}, \label{metric}
\end{align}
where $l(w; (\theta,y)) = \max\{0,1-y\langle w,\theta \rangle\}$ denotes the hinge loss
function, $u, u_i, u_j \in \bbr^n$, $v, v_i, v_j \in \{+1,-1\}$, and $b_{ij}\in \bbr$ are the random
variables satisfying certain probability distributions, and $\lambda, C > 0$ are certain
given parameters. In this problem,  \eqref{svm} is used to
learn the classifier $w$ by using the metric $\bar A$ satisfying
certain requirements in \eqref{metric}, including the low rank (or nuclear norm) assumption.
Problem \eqref{para1}-\eqref{para2} can also be used in some data-driven
applications, where one can use \eqref{para2} to specify the parameters for the probabilistic models
associated with the random variable $\xi$, as well as some other applications for multi-objective
stochastic optimization.

In spite of its wide applicability, the study on efficient solution methods for expectation constrained optimization
is still limited. For the sake of simplicity, suppose for now that $\xi$ is given as a deterministic vector and
hence that the objective functions $f$ and $\phi$ in
\eqref{1.1} and \eqref{para1} are easily computable.
One popular method to solve stochastic optimization problems is called the sample average
approximation (SAA) approach (\cite{sha03,ksh,wang2008sample}).
To apply SAA for \eqref{1.1} and \eqref{para2}, we first generate
a random sample $\xi_i, i = 1, \ldots, N$, for some $N \ge 1$ and then approximate
$g$ by $\tilde g (x)=\tfrac{1}{N}\tsum_{i=1}^N G(x,\xi_i)$.
The main issues associated with the SAA for solving \eqref{1.1} include: i) the deterministic SAA problem
might not be feasible;
ii) the resulting
deterministic SAA problem is often difficult to solve especially when $N$ is large, requiring going through the whole sample $\{\xi_1, \ldots, \xi_N\}$
at each iteration; and ii) it is not applicable to the on-line setting where one needs to update the decision variable whenever a new piece of sample $\xi_i$, $i = 1, \ldots N$, is
collected.

A different approach to solve stochastic optimization problems is called stochastic approximation (SA), which
was initially proposed in a seminal paper by Robbins and Monro\cite{RobMon51-1} in 1951 for solving strongly convex SP problems.
This algorithm mimics the gradient descent method by using the stochastic gradient $F'(x, \xi_i)$ rather than the original gradient $f'(x)$
for minimizing $f(x)$ in \eqref{1.1} over a simple convex set $X$ (see also \cite{BMP87,Erm83-1,Gaiv78-1,Pflug96-1,RusSys86-1,spall2005introduction}).
An important improvement of this algorithm was developed by Polyak and Juditsky(\cite{pol90},\cite{pol92}) through using longer steps and then averaging the obtained iterates. Their method was shown to be more robust with respect to the choice of stepsize than classic SA method for solving strongly convex SP problems.
More recently, Nemirovski et al. \cite{nemirovski2009robust} presented a modified SA method, namely, the mirror descent SA method,
and demonstrated its superior numerical performance for solving a general class of nonsmooth convex SP problems.
The SA algorithms have been intensively studied over the past few years (see, e.g., \cite{Lan10-3,GhaLan12-2a,GhaLan13-1,DSYA10,Xiao10-1,GhaLan13-2,Nedic12-1,SchRouBac13-1}).
It should be noted, however, that none of these SA algorithms are applicable to expectation constrained problems, since
each iteration of these algorithms requires the projection over the feasible set $\{x \in X| g(x) \le 0\}$, which is computationally prohibitive as $g$ is given in the form
of expectation. 

In this paper, we intend to develop efficient solution methods for solving expectation constrained problems
by properly addressing the aforementioned issues associated with existing SA methods.
Our contribution mainly exists in the following several aspects.
Firstly, inspired by Polayk's subgradient method for constrained optimization~\cite{polyak1967general}{\color{black} and Nesterov's note \cite{Nest04}}, we
present a new SA algorithm, namely the cooperative SA (CSA) method for solving the SP problem with
expectation constraint in \eqref{1.1} with constraint \eqref{1.1exp}. At the $k$-th iteration, CSA
performs a projected subgradient step along either $F'(x_k, \xi_k)$ or
$G'(x_k, \xi_k)$ over the set $X$, {\color{black} depending on whether an unbiased estimator $\hat G_k$ of $g(x_k)$ satisfies $\hat G_k \le \eta_k$ or not.
Observe that the aforementioned estimator $\hat G_k$ can be easily computed in many cases by using the structure of the problem, e.g., the linear dependence $\xi^Tx$ in \eqref{CVaR} (see Section 4.1 in \cite{lns11} and Section 2.1  for more details).
We introduce an index set $\B:= \{1 \le k \le N: \hat G_k \le \eta_k\}$ in order to compute the
output solution as a weighted average of the iterates in $\B$}. By carefully bounding $|\B|$, we show that
the number of iterations performed by the CSA algorithm to find an $\epsilon$-solution of \eqref{1.1}, i.e.,
a point $\bar x \in X$ s.t.
$f(\bar{x})-f^*\leq \epsilon$ and $g(\bar x) \le \epsilon$ with high probability, can be bounded by ${\cal O}(1/\epsilon^2)$.
Moreover, when  both $f$ and $g$ are strongly convex, by using a different set of algorithmic parameters
we show that the complexity of the CSA method can
be significantly improved to ${\cal O}(1/\epsilon)$. It it is worth mentioning that this result
is new even for solving deterministic strongly convex problems with function constraints.
We also established the large-deviation properties for the CSA method under certain
light-tail assumptions.

Secondly, we develop a variant of CSA, namely the cooperative stochastic parameter
approximation (CSPA) method for solving the SP problem with expectation constraints
on problem parameters in \eqref{para1}-\eqref{para2}. {\color{black} In CSPA, we update parameter $x$ by running the mirror descend SA iterates
whenever a certain easily verifiable condition is violated.
Otherwise, we update the decision variable $y$ while keeping $x$ intact}. We show that the number of iterations
performed by the CSPA algorithm to find an $\epsilon$-solution of \eqref{para1}-\eqref{para2},
i.e., a pair of solution $(\bar x, \bar y)$ s.t. $g(\bar x) \le \epsilon$ and
$\bbe[\phi(\bar x, \bar y) - \phi(\bar x, y^*(\bar x)] \le \epsilon$ with high probability, can be bounded by
${\cal O}(1/\epsilon^2)$. Moreover, this bound can be significantly improved
to ${\cal O}(1/\epsilon)$ if $G$ and $\Phi$ are strongly convex w.r.t. $x$ and $y$, respectively.

To the best of our knowledge, all the aforementioned algorithmic developments are
new in the stochastic optimization literature. It is also worth mentioning a few alternative or related
methods to solve \eqref{1.1} and \eqref{para1}-\eqref{para2}. First,
without efficient methods to directly solve \eqref{1.1}, current practice resorts to
reformulate it as $\min_{x \in X} \lambda f(x) + (1-\lambda) g(x)$ for some $\lambda \in (0,1)$. However,
one then has to face the difficulty of properly
specifying $\lambda$, since an optimal selection would depend on the unknown dual multiplier.
As a consequence,  we cannot assess the quality of the solutions obtained by solving this reformulated problem.
Second, one alternative approach to solve \eqref{1.1} is the penalty-based or primal-dual
approach. However these methods would require either the estimation of the optimal dual variables or iterations performed
on the dual space (see \cite{CheLanOu13-1}, \cite{nemirovski2009robust} and \cite{LaLuMo11-1}). Moreover,
the rate of convergence of these methods for function constrained problems has not been well-understood other than
conic constraints even for the deterministic setting. Third, in \cite{jiang2014solution} (and see references therein), Jiang and Shanbhag
developed a coupled SA method to solve a stochastic optimization problem with parameters given by
another optimization problem, and hence is not applicable to problem~\eqref{para1}-\eqref{para2}.
Moreover, each iteration of their method requires two stochastic subgradient projection
steps and hence is more expensive than that of CSPA.

The remaining part of this paper is organized as follows. In Section 2, we present the CSA algorithm
and establish its convergence properties under general convexity and strong convexity assumptions.
Then in Section 3, we develop a variant of the CSA algorithm, namely the CSPA for solving SP problems
with the expectation constraint over problem parameters and discuss its convergence properties.
We then present some numerical results for these new SA methods in section 4.
Finally some concluding remarks are added in Section 5.

\setcounter{equation}{0}
\section{function or expectation constraints over decision variables}
In this section we present the cooperative SA (CSA) algorithm for solving convex stochastic optimization problems with the  constraint
over decision variables. More specifically, we first briefly review the
distance generating function and prox-mapping in Subsection 2.1. {\color{black} We then describe the CSA algorithm in Subsection 2.2 and discuss its convergence properties in terms of expectation and large deviation for solving
general convex problems in Subsection 2.3. Then we show how to apply the CSA algorithm to problem \eqref{1.1} with expectation constraint and discuss its large deviation properties in Subsection 2.4. Finally, we show how to improve the convergence of this algorithm by imposing strong convexity assumptions to problem \eqref{1.1} in Subsection 2.5}.

\subsection{Preliminary: prox-mapping}
Recall that a function $\omega_X: X \mapsto R$ is a distance generating function with parameter $\alpha$, if $\omega_X$ is continuously differentiable and strongly convex with parameter $\alpha$ with respect to $\|\cdot\|$. Without loss of generality, we assume throughout this paper that $\alpha=1$, because we can always rescale $\omega_X(x)$ to $\bar{\omega}_X(x)= \omega_X(x)/\alpha$. Therefore, we have
$$\langle x-z, \nabla\omega_X(x) - \nabla\omega_X(z) \rangle \geq \|x-z\|^2, \forall x,z\in X.$$
The prox-function associated with $\omega$ is given by
$$V_X(z,x) = \omega_X(x) - \omega_X(z) - \langle \nabla\omega_X(z), x-z \rangle. $$
$V_X(\cdot,\cdot)$ is also called the Bregman's distance, which was initially studied by Bregman \cite{bregman1967relaxation} and later by many others (see \cite{auslender2006interior},\cite{BBC03-1} and \cite{Teb97-1}). In this paper we assume the prox-function $V_X(x,z)$ is chosen such that, for a given $x\in X$, the prox-mapping $P_{x,X} : \mathbb{R}^n \mapsto\mathbb{R}^n$ defined as
\begin{equation}\label{Prox-mapping def}
    P_{x,X}(\cdot) := \argmin_{z\in X}\{\langle\cdot,z\rangle+V_X(x,z)\}
\end{equation}
is easily computed.

It can be seen from the strong convexity of $\omega(\cdot,\cdot)$ that
\begin{equation}\label{V(,)lowerbound}
    V_X(x,z)\geq \tfrac{1}{2}\|x-z\|^2, \forall x,z \in X.
\end{equation}
Whenever the set $X$ is bounded, the distance generating function $\omega_X$ also gives rise to the diameter of $X$ that will be used frequently in our convergence analysis:
\begin{equation}\label{D_X_def}
D_X\equiv D_{X,\omega_X} := \sqrt{\max_{x, z \in X}V_X(x,z)}.
\end{equation}

The following lemma follows from the optimality condition of \eqref{Prox-mapping def} and the definition
of the prox-function (see the proof in \cite{nemirovski2009robust}).
\begin{lem}\label{Lemma 1}
For every $u,x\in X$, and $y\in \bbr^n$, we have
$$ V_X(P_{x,X}(y),u)\leq V_X(x,u)+ y^T(u-x)+ \tfrac{1}{2}\|y\|_*^2,
$$
where the $\|\cdot\|_*$ denotes the conjugate of $\|\cdot\|$, i.e., $\|y\|_* = \max \{\langle x, y\rangle|\|x\|\leq 1\}$.
\end{lem}

\subsection{The CSA method}
In this section, we present a generic algorithmic framework for solving the constrained optimization problem in \eqref{1.1}. We assume the expectation function $f(x)$ and constraint $g(x)$, in addition to being well-defined and finite-valued for every $x\in X$, are continuous and convex on $X$.

The CSA method can be viewed as
a stochastic counterpart of Polayk's subgradient method, which was originally
designed for solving deterministic nonsmooth convex optimization problems (see \cite{polyak1967general} and a more recent generalization in \cite{beck2010comirror}).
At each iterate $x_k$, $k \ge 0$, depending on whether $g(x_k) \le \eta_k$ for some tolerance $\eta_k >0$, it moves either along
the subgradient direction $f'(x_k)$ or $g'(x_k)$,
with an appropriately chosen stepsize $\gamma_k$ which also depends on $\|f'(x_k)\|_*$ and $\|g'(x_k)\|_*$. However, Polayk's subgradient method cannot be applied
to solve \eqref{1.1} because we do not have access to exact information about $f'$, $g'$ and $g$.
The CSA method differs from
Polyak's subgradient method in the following three aspects. Firstly, the search direction $h_k$ is defined in a stochastic manner: we first check
if the solution $x_k$ we computed at iteration k violates the condition $\hat G_k\leq \eta_k$ for some $\eta_k \geq 0$. If so, we set the $h_k = G'(x_k,\xi_k)$ for a random realization $\xi_k$ of $\xi$ (Note that for deterministic constraint in \eqref{1.1}, $h_k = g'(x_k)$) in order to control the violation of expectation constraint. Otherwise, we set $h_k = F'(x_k,\xi_k)$. Secondly, for some $1\leq s\leq N$, we partition the indices $I=\{s,...,N\}$ into two subsets: $\B = \{s\leq k\leq N| \hat G_k\leq \eta_k\}$ and $\N = I \setminus \B$, and define the output $\bar{x}_{N,s}$ as an ergodic mean of $x_k$ over $\B$. This differs from the Polyak's subgradient method that defines the output solution as the best $x_k, k\in \B$, with the smallest objective value. Thirdly,
while the original Polayk's subgradient method were developed only for general nonsmooth problems, we show that the
CSA method
also exhibits an optimal rate of convergence for solving strongly convex problems by properly choosing $\{\gamma_k\}$ and $\{\eta_k\}$.
\begin{algorithm}
\caption{The cooperative SA algorithm}\label{CSA alg}
\indent\ \ \  {\bf Input:} initial point $x_1\in X$, stepsizes $\{\gamma_k\}$ and tolerances $\{\eta_k\}$.
\begin{algorithmic}
 \State {\bf for} $k = 1, 2, \ldots, N$ \\
  \ \ {\color{black} Let $\hat G_k$ be an unbiased estimator of $g(x_k)$.} Set
  \begin{align}
  h_k=& \left\{
                   \begin{array}{ll}
                     F'(x_k,\xi_k), & \hbox{if {\color{black} $\hat G_k\leq\eta_k$};} \label{f_iter5.1} \\
                     G'(x_k,\xi_k), & \hbox{otherwise.}
                   \end{array}
                   \right.\\
  x_{k+1}= & P_{x_k,X}(\gamma_k h_k ). \label{f_iter5.2}
  \end{align}
  \State {\bf end for}
\State {\bf Output}:
Set {\color{black} $\B =\{s\leq k\leq N| \hat G_k\leq \eta_k\}$} for some $1\leq s\leq N$, and define the output
\begin{equation}\label{Solution5.1}
    \bar{x}_{N,s}=(\tsum_{k\in \mathcal{B}}\gamma_k)^{-1}(\tsum_{k\in \mathcal{B}}\gamma_kx_k),
\end{equation}

\end{algorithmic}
\end{algorithm}

{\color{black} Notice that every iteration of CSA requires an unbiased estimator of $g(x_k)$.
Suppose there is no uncertainty associated with the constraint in \eqref{1.1}, we can evaluate $g(x_k)$ exactly.
If $g$ is given in the form of expectation,  one natural way is to generate a J-sized i.i.d. random sample of $\xi$
and then evaluate the constraint function value by $\hat G_k = \tfrac{1}{J}\tsum_{j=1}^J G(x_k,\xi_j)$. However,
this basic scheme can be much improved by using some structural information for constraint evaluation. For instance,
one ubiquitous structure existing in machine learning and portfolio optimization applications is
the linear combination of $\xi^Tx$.
For a given $x \in X$, we can define a new random variable $\bar \xi= \xi^Tx$ and generate samples of $\bar \xi$ instead of $\xi$. $\bar \xi$ is only of dimension one and it is computationally much cheaper to simulate.
Given the distribution of $\xi$, below we provide a few examples where the distribution of $\bar \xi$ can be explicitly computed or approximated.
For instance, if $x\in \mathbb{R}^d$, $\xi_i$ are independent normal $N(\mu_i,\sigma_i)$, then $\bar\xi$ follows $N(\tsum_{i=1}^d \mu_i,[\tsum_{i=1}^d x_i^2\sigma_i^2]^{1/2})$. If $\xi_i$ follows independent $\exp(\lambda_i)$, then the probability density function of $\bar \xi$ is $$
f_{\bar \xi}(y) = (\prod_{i=1}^d\hat\lambda_i)\tsum_{j=1}^d \tfrac{e^{-\hat\lambda_j}y}{\prod_{k\neq j,k=1}^d (\hat \lambda_k\hat \lambda_j)},$$
where $\hat \lambda_i = \lambda_i/x_i$. If $\xi_i$ follows independent $\text{Uniform}(a,b)$, then the cumulative distribution function of $\bar \xi$ is
\begin{align*}
F_{\bar \xi}(y) = & \tfrac{1}{d!\prod_{i=1}^dx_i}\{(\tfrac{y-a\tsum_{i=1}^dx_i}{b-a}^+)^d  + \tsum_{v=1}^d(-1)^v
\tsum_{j_1=1}^d \tsum_{j_2 = j_1+1}^d \ldots \\
&\tsum_{j_v = j_{v-1}+1}^d \{[\tfrac{y-a\tsum_{i=1}^dx_i}{b-a}-(x_{j_1}+x_{j_2}+\ldots+x_{j_v})]^+\} \}.
\end{align*}
If the $\xi_i$ are dependent normal random variables with mean $\mu$ and covariance $C$ (by Cholesky decomposition, $C=LL'$),
we can estimate $\tsum_{i=1}\xi_ix_i$ by
$\tsum_{i=1}^d \mu_ix_i+ \bar r [\tsum_{i=1}^d(L^Tx)_i^2]^{1/2}$,
where $\bar r$ follows $N(0,1)$.
In fact, when the dimension $d$ is large enough, by central limit theorem, we can use a normal distribution to approximate the new random variable $\bar \xi$.
These are a few examples showing that to simulate $\bar \xi$ can be much faster than to simulate the original random variables for constraint evaluation.
}

\subsection{Convergence of CSA for SP with function constraints}
In this subsection, we consider the case when the constraint function $g$ is deterministic (i.e., $\hat G_k = g'(x_k)$).
Our goal is to establish the rate of convergence associated with CSA, in terms of both the distance to the optimal value and the violation of constraints. It should also be noted that Algorithm~\ref{CSA alg} is conceptional only as we have not specified a few algorithmic parameters (e.g. $\{\gamma_k\}$ and $\{\eta_k\}$). We will come back to this issue after establishing some general properties about this method. Throughout this subsection, we make the following assumptions.\newline

\begin{assumption}
For any $x\in X$, a.e. $\xi \in {\cal P}$,
$$\mathbb{E}[\|F'(x,\xi)\|_*^2]\leq M_F^2\ \ \text{and}\ \ \ \|g'(x)\|_*^2\leq M_G^2,$$
where $F'(x,\xi) \in \partial_x F(x,\xi)$ and $g'(x) \in \partial_x g(x)$.
\end{assumption}

The following result establishes a simple but important recursion about the CSA method for problem~\eqref{1.1}.
\begin{prop}\label{Proposition 1}
For any $1\leq s\leq N$, we have
\begin{equation}\label{2.1}
    \tsum_{k\in \N}\gamma_k(\eta_k - g(x))+\tsum_{k\in \B}\gamma_k\langle F'(x_k,\xi_k),x_k-x\rangle\leq V(x_s,x)+\tfrac{1}{2}\tsum_{k\in\B}\gamma_k^2\| F'(x_k,\xi_k)\|_*^2+\tfrac{1}{2}\tsum_{k\in\N}\gamma_k^2\| g'(x_k)\|_*^2,
\end{equation}
for all $x\in X$.
\end{prop}
\begin{proof}
For any $s\leq k\leq N$, using Lemma \ref{Lemma 1}, we have
\begin{equation}\label{Vrecursion}
     V(x_{k+1},x)\leq V( x_k,x)+\gamma_k\langle h_k,x-x_k\rangle+\tfrac{1}{2}\gamma_k^2\| h_k\|_*^2.
\end{equation}
Observe that if $k\in \B$, we have $h_k=F'(x_k,\xi_k),$ and
$$\langle h_k,x_k-x\rangle=\langle F'(x_k,\xi_k),x_k-x\rangle.$$
Moreover, if $k\in \mathcal{N}$, we have $h_k=g'(x_k)$ and
$$
\begin{aligned}
 \langle h_k,x_k-x\rangle& = \langle g'(x_k),x_k-x\rangle
  \geq g(x_k) - g(x)
  \geq \eta_k - g(x).
\end{aligned}
$$
Summing up the inequalities in \eqref{Vrecursion} from $k=s$ to $N$ and using the previous two observations, we obtain
\begin{equation}
\begin{aligned}
    V(x_{k+1},x) &\leq V(x_s,x)-\tsum_{k=s}^N\gamma_k\langle h_k,x_k-x\rangle+\tfrac{1}{2}\tsum_{k=s}^N\gamma_k^2\| h_k\|_*^2 \\
 &\leq  V(x_s,x)-\left[\tsum_{k\in \N}\gamma_k\langle g'(x_k), x_k-x\rangle+\tsum_{k\in \B}\gamma_k\langle F'(x_k,\xi_k),x_k-x\rangle\right]+\tfrac{1}{2}\tsum_{k=s}^N\gamma_k^2\| h_k\|_*^2 \\
    & \leq V(x_s,x)-\left[\tsum_{k\in\N}\gamma_k(\eta_k - g(x))+\tsum_{k\in \B}\gamma_k\langle F'(x_k,\xi_k),x_k-x\rangle\right]\\
    & \quad +\tfrac{1}{2}\tsum_{k\in\B}\gamma_k^2\| F'(x_k,\xi_k)\|_*^2 +\tfrac{1}{2}\tsum_{k\in\N}\gamma_k^2\| g'(x_k)\|_*^2.
\end{aligned}
\end{equation}
Rearranging the terms in above inequality, we obtain \eqref{2.1}
\end{proof}

\vgap

Using Proposition \ref{Proposition 1}, we present below a sufficient condition under which the output solution $\bar{x}_{N,s}$ is well-defined.
\begin{lem}\label{Lemma 2}
Let $x^*$ be an optimal solution of \eqref{1.1}. If, for some $\rho\in (0,1)$, we have
\begin{equation}\label{2.2}
    \tfrac{N-s+1}{2}\min_{k\in \N}\gamma_k\eta_k >\tfrac{1}{\rho} \left(D_X^2+\tfrac{M^2}{2}\tsum_{k=s}^N\gamma_k^2\right),
\end{equation}
where $M = \max\{M_F, M_G\}$, 
then with probability at least $1-\rho$, we have 
\begin{equation}\label{cond_prop}
\tsum_{k\in\mathcal{N}}\gamma_k[\eta_k-g(x^*)]+\tsum_{k\in \mathcal{B}}\gamma_k\langle f'(x_k),x_k-x^*\rangle 
\leq  \tfrac{1}{\rho}\left(D_X^2+\tfrac{M^2}{2}\tsum_{k=s}^N\gamma_k^2\right).
\end{equation}
Moreover,  suppose that \eqref{cond_prop} holds. Then, $\mathcal{B}\neq \emptyset$, i.e., $\bar{x}_{N,s}$ is well-defined, and we have one of the following two statements holds,
\begin{description}
  \item[a)] $|\mathcal{B}| \geq (N-s+1)/2,$
  \item[b)] $
  \tsum_{k\in \B}\gamma_k\langle f'(x_k),x_k-x^*\rangle \leq 0.$
\end{description}
\end{lem}
\begin{proof}
Taking expectation w.r.t. $\xi_k$ on both sides of \eqref{2.1} and fixing $x = x^*$, we have 
\begin{align*}
& \bbe[\tsum_{k\in \N}\gamma_k(\eta_k - g(x^*))+\tsum_{k\in \B}\gamma_k\langle F'(x_k,\xi_k),x_k-x^*\rangle] \\
&=  \tsum_{k=s}^N \bbe_{\xi_{[k-1]}} \left\{ \bbe_{\xi_k}\left[ \gamma_k[\eta_k-g(x^*)](1-I_{\mathcal{B}}(k))+\gamma_k\langle F'(x_k,\xi_k),x_k-x^*\rangle I_{\mathcal{B}}(k) | \xi_{[k-1]} \right] \right\} \\
& \leq  V(x_s,x^*)+\tfrac{M^2}{2}\tsum_{k=s}^N\gamma_k^2,
\end{align*}
where $I_{\mathcal{B}}(k) = 1$ if $k\in\mathcal{B}$, otherwise $I_{\mathcal{B}}(k) = 0$. Since $I_{\mathcal{B}}(k)$ is independent of $\xi_k$,
\begin{equation}\label{2.2proof}
\begin{aligned}
&   \bbe\left[ \tsum_{k\in\mathcal{N}}\gamma_k[\eta_k-g(x^*)]+\tsum_{k\in \mathcal{B}}\gamma_k\langle f'(x_k),x_k-x^*\rangle \right]   \leq  D_X^2+\tfrac{M^2}{2}\tsum_{k=s}^N\gamma_k^2.
\end{aligned}
\end{equation}
Now by \eqref{2.2proof} and the Markov inequality, we have, for any $\rho\in (0,1)$,
$$
\Prob\left\{\tsum_{k\in\mathcal{N}}\gamma_k[\eta_k-g(x^*)]+\tsum_{k\in \mathcal{B}}\gamma_k\langle f'(x_k),x_k-x^*\rangle 
\geq  \tfrac{1}{\rho}\left(D_X^2+\tfrac{M^2}{2}\tsum_{k=s}^N\gamma_k^2\right)\right\} \leq \rho.
$$
Hence, with probability at least $1-\rho$, we have  \eqref{cond_prop}.
Now assume that \eqref{cond_prop} holds. Suppose for contradiction that $\mathcal{B}= \emptyset$. We then conclude from \eqref{cond_prop}
and the fact $g(x^*)\leq 0$ that
$$
(N-s+1)\min_{k\in \N}\gamma_k\eta_k\leq  \tsum_{k=s}^N\gamma_k[\eta_k-g(x^*)] \leq  \tfrac{1}{\rho}\left(D_X^2+\tfrac{M^2}{2}\tsum_{k=s}^N\gamma_k^2\right),
$$
which contradicts with \eqref{2.2}. Hence, we must have $\mathcal{B}\neq \emptyset$. If $\tsum_{k\in \B}\gamma_k\langle f'(x_k),x_k-x^*\rangle \leq 0$, part b) holds. Otherwise,
if $\tsum_{k\in \B}\gamma_k\langle f'(x_k),x_k-x^*\rangle \geq 0$, we have
$$ \tsum_{k\in \mathcal{N}}\gamma_k[\eta_k-g(x^*)]\leq  \tfrac{1}{\rho}(D_X^2+\tfrac{M^2}{2}\tsum_{k=s}^N\gamma_k^2),$$
which, in view of $g(x^*)\leq 0$, implies that
\begin{equation}\label{2.4}
 \tsum_{k\in \mathcal{N}}\gamma_k\eta_k\leq  \tfrac{1}{\rho}(D_X^2+\tfrac{M^2}{2}\tsum_{k=s}^N\gamma_k^2).
\end{equation}
Suppose that $|\mathcal{B}| < (N-s+1)/2$, i.e., $| \mathcal{N}| \geq (N-s+1)/2$. Then,
$$\tsum_{k\in \mathcal{N}}\gamma_k\eta_k \geq \tfrac{N-s+1}{2}\min_{k\in \N}\gamma_k\eta_k > \tfrac{1}{\rho}(D_X^2+ \tfrac{M^2}{2}\tsum_{k=s}^N\gamma_k^2),$$
which contradicts with \eqref{2.4}. Hence, part a) holds.
\end{proof}

Now we are ready to establish the main convergence properties of the CSA method. 
\begin{thm}\label{Theorem 1}
For some $\rho\in (0,1)$, suppose that $\{\gamma_k\}$ and $\{\eta_k\}$ in the CSA algorithm are chosen such that \eqref{2.2} holds. Then with probability 
at least $1-\rho$, for any $1\leq s\leq N$, we have
\begin{align}
 & f(\bar{x}_{N,s})-f(x^*) \leq \tfrac{2D_X^2+ M^2\tsum_{s\leq k\leq N}\gamma_k^2}
 {\rho(N-s+1)\min_{s\leq k\leq N}\gamma_k}, \label{2.3.2}\\
 &g(\bar{x}_{N,s})  \leq (\tsum_{k\in \B}\gamma_k)^{-1}(\tsum_{k\in \B}\gamma_k\eta_k),\label{2.3.1}
 \end{align}
 where $M := \max\{M_F,M_G\}$.
\end{thm}

\begin{proof}
We first show \eqref{2.3.2}. By Lemma \ref{Lemma 2}, we have \eqref{cond_prop} holds with probability at least $1-\rho$. 
Now suppose \eqref{cond_prop} holds. If Lemma \ref{Lemma 2} part (b) holds, dividing both sides by $\tsum_{k\in\mathcal{B}}\gamma_k$, using the convexity of $f$ and Jensen's inequality, we have
\begin{equation}\label{f_bound1}
f(\bar{x}_{N,s})-f(x^*)\leq 0.
\end{equation}
If $|\B|\geq (N-s+1)/2$, we have
$\tsum_{k\in \mathcal{B}}\gamma_k\geq |\B|\min_{k\in \B}\gamma_k \geq \tfrac{N-s+1}{2}\min_{k\in \B}\gamma_k.
$
It follows from the fact $\gamma_k\eta_k \geq 0, g(x^*)\leq 0$ and \eqref{cond_prop} that
$$
\tsum_{k\in \mathcal{B}}\gamma_k( f(x_k)- f(x^*)) 
\leq  \tfrac{1}{\rho}\left(D_X^2+\tfrac{M^2}{2}\tsum_{k=s}^N\gamma_k^2\right)
$$
which, in view of the definition of $\bar{x}_{N,s}$ in \eqref{Solution5.1} and the convexity of $f(\cdot)$, implies that
\begin{equation}\label{f_bound2}
\begin{aligned}
f(\bar{x}_{N,s})-f(x^*) & \leq \tfrac{2D_X^2+M^2\tsum_{s\leq k\leq N}\gamma_k^2}{ \rho(N-s+1)\min_{k\in \mathcal{B}}\gamma_k}.
\end{aligned}
\end{equation}
Combining these two inequalities \eqref{f_bound1} and \eqref{f_bound2}, we have \eqref{2.3.2}.
Now we show that \eqref{2.3.1} holds. For any $k\in \mathcal{B}$, we have $g(x_k)\leq \eta_k$. This observation, in view of the definition of $\bar{x}_{N,s}$ in \eqref{Solution5.1}, the convexity of $g(\cdot)$ and Jensen's inequality, then implies that
\beq \label{constraint_proof}
\begin{aligned}
   g(\bar{x}_{N,s})  &\leq  \tfrac{\tsum_{k\in \mathcal{B}}\gamma_kg(x_k)}{\tsum_{k\in \mathcal{B}}\gamma_k}
     \leq \tfrac{\tsum_{k\in \mathcal{B}}\gamma_k\eta_k}{\tsum_{k\in \mathcal{B}}\gamma_k}.
\end{aligned}
\eeq
\end{proof}
Below we provide a few specific selections of $\{\gamma_k\}$, $\{\eta_k\}$ and $s$ to establish the rate of convergence for the CSA method. In particular, we will present a constant and variable stepsize policy, respectively, in Corollaries~\ref{Coro 1} and \ref{Coro 2}.
\begin{cor}\label{Coro 1}
If s=1,$\gamma_k=\tfrac{D_X}{\sqrt{N}M}$ and $\eta_k=\tfrac{4M D_X}{ \rho\sqrt{N}}$, $k=1,...N$,
then, with probability at least $1-\rho$, we have
\begin{align*}
f(\bar{x}_{N,s})-f(x^*) &\leq \tfrac{4D_XM}{ \rho\sqrt{N}},\\
g(\bar{x}_{N,s})  &\leq\tfrac{4D_XM}{ \rho\sqrt{N}}.
\end{align*}
\end{cor}
\begin{proof}
First, observe that condition \eqref{2.2} holds by using the facts that
$$
\begin{aligned}
&\tfrac{N-s+1}{2}\min_{k\in\N}\gamma_k\eta_k = \tfrac{N}{2}\tfrac{4D_X^2}{\rho N} = \tfrac{2D_X^2}{ \rho}, \\
&D_X^2+\tfrac{M^2}{2}\tsum_{k=s}^N\gamma_k^2 = D_X^2+\tfrac{1}{2}\tsum_{k=1}^N \tfrac{D_X^2}{N} \leq 2D_X^2.
\end{aligned}$$
It then follows from Lemma \ref{Lemma 2} and Theorem \ref{Theorem 1} that
\begin{align*}
& f(\bar{x}_{N,s})-f(x^*) \leq\tfrac{4D_XM}{  \rho\sqrt{N}},\\
&g(\bar{x}_{N,s}) \leq \max_{s\leq k\leq N} \eta_k =  \tfrac{4D_XM}{ \rho\sqrt{N}}.
\end{align*}
\end{proof}

\begin{cor}\label{Coro 2}
If $s=\tfrac{N}{2}$,
$\gamma_k=\tfrac{D_X}{\sqrt{k}M}$ and $\eta_k=\tfrac{4D_XM}{ \rho\sqrt{k}},$ $k=1,2,...,N$, then with probability at least $1-\rho$, we have
\begin{align*}
 f(\bar{x}_{N,s})-f(x^*) &\leq \tfrac{4D_X(1+\tfrac{1}{2}\log2)M}{ \rho\sqrt{N}},\\
g(\bar{x}_{N,s})  &\leq \tfrac{4\sqrt{2} D_XM}{  \rho \sqrt{N}}.
\end{align*}
\end{cor}
\begin{proof}
The proof is similar to that of corollary 4 and hence the details are skipped.
\end{proof}

\vgap

In view of Corollaries~\ref{Coro 1} and \ref{Coro 2}, the CSA algorithm achieves an ${\cal O}(1/\sqrt{N})$ rate of
convergence for solving problem~\eqref{1.1}. This convergence rate seems to be unimprovable as it matches the
optimal rate of convergence for deterministic convex optimization problems with function constraints~\cite{Nest04}.
However,
to the best of our knowledge, no such complexity bounds have been obtained before for solving
stochastic optimization problems with function constraints.

%
It follows from  Corollary~\ref{Coro 1} and \ref{Coro 2} that in order to find a solution $\bar{x}_{N,s}\in \setU$ such that
$$\Prob \left\{f(\bar{x}_{N,s})-f(x^*)\leq \epsilon, g(\bar{x}_{N,s})\leq \epsilon\right\}\ge 1-\Lambda,$$
the number of iteration performed by the CSA method can be bounded by
\begin{equation}\label{Com1}
{\cal O}\left\{\tfrac{1}{\epsilon^2\Lambda^2}\right\}.
\end{equation}
We will show that this result can be significantly improved if Assumption A1 is augmented by the following ``light-tail" assumption,
which is satisfied by a wide class of distributions (e.g., Gaussian and t-distribution).
\begin{assumption}\label{light_a}
For and $x \in X$,
\begin{align*}
\bbe[\exp\{\| F'(x,\xi)\|_*^2/M_F^2\}] &\leq \exp\{1\}.
\end{align*}
\end{assumption}

We first present the following Bernstein inequality that will be used to establish the large-deviation properties of the CSA method (e.g. see \cite{nemirovski2009robust}).
Note that in the sequel, we  denote $\xi_{[k]} := \{\xi_1,\ldots,\xi_k\}.$
\begin{lem}\label{Lemma 5}
Let $\xi_1,\xi_2,...$ be a sequence of i.i.d. random variables, and $\xi_t=\xi(\xi_{[t]})$ be deterministic Borel functions of $\xi_{[t]}$ such that $\bbe[\xi_t]=0$ a.s. and $\bbe[\exp\{\xi_t^2/\sigma_t^2\}]\leq \exp\{1\}$ a.s., where $\sigma_t> 0$ are deterministic. Then
$$\forall \lambda\geq 0: \Prob\left\{\tsum_{t=1}^N\xi_t>\lambda\sqrt{\tsum_{t=1}^N\sigma_t^2}\right\}\leq \exp\{-\lambda^2/3\}.$$
\end{lem}

Now we are ready to establish the large deviation properties of the CSA algorithm.
\begin{thm}\label{Theorem 3}
Under Assumption 2, $\forall \lambda \geq 0$,
\begin{align}
    \Prob\{f(\bar{x}_{N,s})-f(x^*)\geq K_0 +\lambda K_1\}\leq \exp\{-\lambda\}+\exp\{-\tfrac{\lambda^2}{3}\},\label{2.6}
\end{align}
where $K_0=\tfrac{\tfrac{1}{2}D_X^2+M_F^2\tsum_{k\in \B}\gamma_k^2+M_G^2\tsum_{k\in \N}\gamma_k^2}{\tsum_{k\in \B}\gamma_k}$ and \\ $K_1=\tfrac{M_F^2\tsum_{k\in \B}\gamma_k^2+M_G^2\tsum_{k\in \N}\gamma_k^2+M_FD_X\sqrt{\tsum_{k=s}^N\gamma_k^2}}{\tsum_{k\in \B}\gamma_k}$.
\end{thm}
\begin{proof}
Let $F'(x_k,\xi_k)=f'(x_k)+\Delta_k$.
It follows from the inequality \eqref{2.1} (with $x=x^*$) and the fact $g(x^*)\leq 0$ that
\begin{multline}\label{ProbabI}
    \tsum_{k\in \N}\gamma_k\eta_k+(\tsum_{k\in \B}\gamma_k)(f(\bar{x}_{N,s})-f(x^*))\leq D_X^2+\tsum_{k\in \B}\gamma_k^2\| F'(x_k,\xi_k)\|_*^2 \\
    +\tsum_{k\in \N}\gamma_k^2\| g'(x_k)\|_*^2 -\tsum_{k\in \B}\gamma_k\langle\Delta_k,x_k-x^*\rangle.
\end{multline}
Now we provide probabilistic bounds for $\tsum_{k\in \B}\gamma_k^2\| F'(x_k,\xi_k)\|_*^2$ and $\tsum_{k\in \B}\gamma_k\langle\Delta_k,x_k-x^*\rangle$. First, setting $\theta_k=\gamma_k^2/\tsum_{k\in \B}\gamma_k^2$, using the fact that $\bbe[\exp\{\| F'(x_k,\xi_k)\|_*^2/M_F^2\}]\leq \exp\{1\}$ and Jensens inequality, we have
$$\exp\{\tsum_{k\in\B}\theta_k(\| F'(x_k,\xi_k)\|_*^2/M_F^2)\}\leq \tsum_{k\in \B}\theta_k\exp\{\| F'(x_k,\xi_k)\|_*^2/M_F^2\},
$$ and hence that
$$\bbe[\exp\{\tsum_{k\in \B}\gamma_k^2\| F'(x_k,\xi_k)\|_*^2/M_F^2\tsum_{k\in\B}\gamma_k^2\}]\leq \exp\{1\}.$$
It then follows from Markov's inequality that $\forall \lambda\geq 0$,
\begin{equation}\label{part1}
\begin{aligned}
   &\Prob(\tsum_{k\in \B}\gamma_k^2\| F'(x_k,\xi_k)\|_*^2> (1+\lambda)M_F^2\tsum_{k\in\B}\gamma_k^2) \\
   &=\Prob\left(\exp\left\{\tfrac{\tsum_{k\in \B}\gamma_k^2\| F'(x_k,\xi_k)\|_*^2}{M_F^2\tsum_{k\in\B}\gamma_k^2}\right\}> \exp(1+\lambda)\right) \\
   &\leq \tfrac{\exp\{1\}}{\exp\{1+\lambda\}} \leq \exp\{-\lambda\}.
\end{aligned}
\end{equation}

Then, let us consider $\tsum_{k\in \B}\gamma_k\langle\Delta_k,x_k-x^*\rangle$. Setting $\beta_k=\gamma_k\langle\Delta_k,x_k-x^*\rangle$, we have 
$\tsum_{k\in \B}\beta_k = \tsum_{k=s}^N \beta_kI_\B(k)$, and $\bbe[\beta_k I_\B(k) | \xi_{[k-1]}] = 0$.
Here $I_{\mathcal{B}}(k) = 1$ if $k\in\mathcal{B}$, otherwise $I_{\mathcal{B}}(k) = 0$.
Also noting that $\mathbb{E}[\|\Delta_k\|_*^2]\leq (2M_F)^2$, we have
$$\bbe[\exp\{\beta_k^2/(2M_F\gamma_k D_X)^2\}]\leq \exp\{1\},$$
which, in view of Lemma \ref{Lemma 5}, implies that
\begin{equation}\label{part3}
    \Prob\left\{\tsum_{k\in \B}\beta_k>2\lambda M_FD_X\sqrt{\tsum_{k=s}^N\gamma_k^2}\right\}\leq \exp\{-\lambda^2/3\}.
\end{equation}
Combining \eqref{part1} and \eqref{part3}, and rearranging the terms we get \eqref{2.6}.
\end{proof}
Applying the stepsize strategy in Corollary 5 to Theorem~{\color{black}\ref{Theorem 3}}, then it follows that the number of iterations performed by the CSA method can be bounded by
$${\cal O}\left\{\tfrac{1}{\epsilon^2}(\log\tfrac{1}{\Lambda})^2\right\}.$$
We can see that the above result significantly improves the one in \eqref{Com1}.

\subsection{{\color{black} Convergence of CSA for SP with expectation constraints}} \label{sec_expconstr}
In this subsection, we focus on the SP problem \eqref{1.1}-\eqref{1.1exp} with the expectation constraint.
We assume the expectation functions $f(x)$ and $g(x)$, in addition to being well-defined and finite-valued for every $x\in X$, are continuous and convex on $X$.
Throughout this section, we assume the Assumption~\ref{light_a} holds. Moreover, with a little abuse of notation, we make the following assumption.
\begin{assumption}\label{light_a1}
for any $x \in X$,
\begin{align}
\bbe[\exp\{\| G'(x,\xi)\|_*^2/M_G^2\}] &\leq \exp\{1\}, \label{light_a1_1}\\
\bbe[\exp\{(G(x,\xi)-g(x))^2/\sigma^2\}] &\leq \exp\{1\}. \label{light_a1_2}
\end{align}
\end{assumption}
{\color{black}We will use \eqref{light_a1_1} and \eqref{light_a1_2} to bound the error associated with stochastic subgradient
and function value for the constraint $g$, respectively. As discussed  in subsection 2.2, there may exist different ways to simulate the random variable $\xi$
for constraint evaluation, e.g.,  by generating a J-sized i.i.d. random sample of $\xi$ or its linear transformation $\bar \xi = \xi^T x$. However,
regardless of the way to simulate the random variable $\xi$,
the light-tail assumption \eqref{light_a1_2} holds for the constraint value $G(x,\xi)$.
Our goal in this subsection is to show how the sample size (or iteration count) $N$ to compute stochastic subgradients,
 as well as the sample size $J$ to evaluate the constraint value, will affect the quality of the solutions generated by CSA.
}

The following result establishes a simple but important recursion about the CSA method for stochastic optimization with expectation constraints.
\begin{prop}\label{Proposition 5.1}
For any $1\leq s\leq N$, we have
\begin{equation}\label{5.2}
\begin{array}{l}
    \tsum_{k\in \N}\gamma_k(G(x_k,\xi_k) - G(x,\xi_k))+\tsum_{k\in \B}\gamma_k\langle F'(x_k,\xi_k),x_k-x\rangle \\
    \quad \leq V(x_s,x)+\tfrac{1}{2}\tsum_{k\in\B}\gamma_k^2\| F'(x_k,\xi_k)\|_*^2+\tfrac{1}{2}\tsum_{k\in\N}\gamma_k^2\| G'(x_k,\xi_k)\|_*^2, \, \forall x \in X.
    \end{array}
\end{equation}
\end{prop}
\begin{proof}
For any $s\leq k\leq N$, using Lemma \ref{Lemma 1}, we have
\begin{equation}\label{Vrecursion5}
     V(x_{k+1},x)\leq V( x_k,x)+\gamma_k\langle h_k,x-x_k\rangle+\tfrac{1}{2}\gamma_k^2\| h_k\|_*^2.
\end{equation}
Observe that if $k\in \B$, we have $h_k=F'(x_k,\xi_k),$ and
$$\langle h_k,x_k-x\rangle=\langle F'(x_k,\xi_k),x_k-x\rangle.$$
Moreover, if $k\in \mathcal{N}$, we have $h_k=G'(x_k,\xi_k)$ and
$$
\begin{aligned}
 \langle h_k,x_k-x\rangle& = \langle G'(x_k,\xi_k),x_k-x\rangle
  \geq G(x_k,\xi_k) - G(x,\xi_k).
\end{aligned}
$$
Summing up the inequalities in \eqref{Vrecursion5} from $k=s$ to $N$ and using the previous two observations, we obtain
\begin{equation}
\begin{aligned}
    V(x_{k+1},x) &\leq V(x_s,x)-\tsum_{k=s}^N\gamma_k\langle h_k,x_k-x\rangle+\tfrac{1}{2}\tsum_{k=s}^N\gamma_k^2\| h_k\|_*^2 \\
 &\leq  V(x_s,x)-\left[\tsum_{k\in \N}\gamma_k\langle G'(x_k,\xi_k), x_k-x\rangle+\tsum_{k\in \B}\gamma_k\langle F'(x_k,\xi_k),x_k-x\rangle\right]+\tfrac{1}{2}\tsum_{k=s}^N\gamma_k^2\| h_k\|_*^2 \\
    &= V(x_s,x)-\left[\tsum_{k\in\N}\gamma_k(G(x_k,\xi_k) - G(x,\xi_k))+\tsum_{k\in \B}\gamma_k\langle F'(x_k,\xi_k),x_k-x\rangle\right]\\
    & \quad +\tfrac{1}{2}\tsum_{k\in\B}\gamma_k^2\| F'(x_k,\xi_k)\|_*^2 +\tfrac{1}{2}\tsum_{k\in\N}\gamma_k^2\| G'(x_k,\xi_k)\|_*^2.
\end{aligned}
\end{equation}
Rearranging the terms in above inequality, we obtain \eqref{5.2}.
\end{proof}

\vgap

Using Proposition \ref{Proposition 5.1}, we present below a sufficient condition under which the output solution $\bar{x}_{N,s}$ is well-defined.
\begin{lem}\label{Lemma 5.2}
Let $x^*$ be an optimal solution of \eqref{1.1}-\eqref{1.1exp}. Under Assumption~\ref{light_a1}, for any given $\lambda >0$, we have
\begin{equation}\label{cond_prop2}
\Prob\{\tsum_{k\in\mathcal{N}}\gamma_k\eta_k+\tsum_{k\in \mathcal{B}}\gamma_k\langle f'(x_k),x_k-x^*\rangle \leq K_0+\lambda K_1\} \geq 1-2\exp\{-\lambda\} - (N-s+2)\exp\{-\lambda^2/3\},
\end{equation}
where $K_0 = D_X^2 + (M_G^2+M_F^2)\tsum_{k=s}^N\gamma_k^2$, $K_1 =  (M_G^2+M_F^2)\tsum_{k=s}^N\gamma_k^2+2(M_G+M_F)D_X\sqrt{\tsum_{k=s}^N\gamma_k^2}+\tfrac{\sigma}{\sqrt{J}}\tsum_{k=s}^N\gamma_k$, and  $J$ is the number of random samples to estimate $g(x_k)$ in each iteration.
If, for some $\lambda>0$,
\begin{equation}\label{5.3}
    \tfrac{N-s+1}{2}\min_{k\in \N}\gamma_k\eta_k >  K_0+\lambda K_1,
\end{equation}
then with probability at least $1-2\exp\{-\lambda\} - (N-s+2)\exp\{-\lambda^2/3\}$, we have one of the following two statements holds,
\begin{description}
  \item[a)]  $|\mathcal{B}| \geq (N-s+1)/2,$
  \item[b)] $
  \tsum_{k\in \B}\gamma_k\langle f'(x_k),x_k-x^*\rangle \leq 0.$
\end{description}
\end{lem}
\begin{proof}
From \eqref{5.2}, 
fixing $x = x^*$, we have
\begin{align*}
& \tsum_{k\in\mathcal{N}}\gamma_k[g(x_k)-g(x^*)]+\tsum_{k\in \mathcal{B}}\gamma_k\langle f'(x_k),x_k-x^*\rangle \leq  D_X^2 + \tfrac{1}{2}\tsum_{k\in\B}\gamma_k^2\|F'(x_k,\xi_k)\|^2 + \tfrac{1}{2}\tsum_{k\in\N}\gamma_k^2\|G'(x_k,\xi_k)\|^2 \\
& \quad \quad+\tsum_{k\in \N}\gamma_k\langle g'(x_k)- G'(x_k,\xi_k),x_k-x\rangle +\tsum_{k\in \B}\gamma_k\langle f'(x_k) - F'(x_k,\xi_k),x_k-x\rangle \\
&\text{then using the facts } \hat G_k \geq \eta_k \text{ for } k\in \N, \text{ and } g(x^*) \leq 0, \\
&\tsum_{k\in\mathcal{N}}\gamma_k\eta_k+\tsum_{k\in \mathcal{B}}\gamma_k\langle f'(x_k),x_k-x^*\rangle \leq  D_X^2 + \tfrac{1}{2}\tsum_{k\in\B}\gamma_k^2\|F'(x_k,\xi_k)\|^2 + \tfrac{1}{2}\tsum_{k\in\N}\gamma_k^2\|G'(x_k,\xi_k)\|^2 \\
& \quad \quad+\tsum_{k\in \N}\gamma_k\langle g'(x_k)- G'(x_k,\xi_k),x_k-x\rangle +\tsum_{k\in \B}\gamma_k\langle f'(x_k) - F'(x_k,\xi_k),x_k-x\rangle +\tsum_{k\in \N}\gamma_k[\hat G_k - g(x_k)].
\end{align*}
Let $\theta_k = \gamma_k^2/\tsum_{k=s}^N \gamma_k^2$, from Jensens inequality, we have
$$\exp\{\tsum_{k=s}^N \theta_k(\|F'(x_k,\xi_k)\|^2/M_F^2)\}\leq \tsum_{k=s}^N \theta_k\exp\{\|F'(x_k,\xi_k)\|^2/M_F^2\},$$
and hence combining Assumption~\ref{light_a} that
$$\bbe[\exp\{\tsum_{k=s}^N\gamma_k^2\|F'(x_k,\xi_k)\|^2/M_F^2\tsum_{k=s}^N\gamma_k^2\}]\leq \exp\{1\}.$$
It then follows from Markov's inequality that $\forall \lambda\geq 0$,
\begin{equation}\label{L10_1}
\begin{aligned}
   &\Prob(\tsum_{k\in \B}\gamma_k^2\| F'(x_k,\xi_k)\|_*^2> (1+\lambda)M_F^2\tsum_{k=s}^N\gamma_k^2) \\
   &\leq \Prob(\tsum_{k=s}^N\gamma_k^2\| F'(x_k,\xi_k)\|_*^2> (1+\lambda)M_F^2\tsum_{k=s}^N\gamma_k^2) \\
   &=\Prob\left(\exp\left\{\tfrac{\tsum_{k =s}^N\gamma_k^2\| F'(x_k,\xi_k)\|_*^2}{M_F^2\tsum_{k=s}^N\gamma_k^2}\right\}> \exp(1+\lambda)\right) \leq \tfrac{\exp\{1\}}{\exp\{1+\lambda\}} \leq \exp\{-\lambda\}.
\end{aligned}
\end{equation}
Similarly, we have
\begin{equation}\label{L10_2}
\Prob(\tsum_{k\in \N}\gamma_k^2\| G'(x_k,\xi_k)\|_*^2> (1+\lambda)M_G^2\tsum_{k=s}^N\gamma_k^2)\leq \exp\{-\lambda\}.
\end{equation}
Let $\delta_k = f'(x_k) - F'(x_k,\xi_k)$, we have
\begin{align*}
\bbe[\langle \delta_k,x_k-x\rangle I_\B(k)|\xi_{[k-1]}] = & \bbe[\langle\delta_k,x_k-x\rangle I_{\B}(k)|\xi_{[k-1]},I_\B(k)=1]\Prob\{I_\B(k) = 1\} \\
& + \bbe[\langle\delta_k,x_k-x\rangle I_{\B}(k)|\xi_{[k-1]},I_\B(k)=0]\Prob\{I_\B(k) = 0\} = 0,
\end{align*}
where $I_\B(k) = 1$ if $k\in\B$, and $I_\B(k) = 0$ otherwise.
So $\langle \delta_k,x_k-x\rangle I_\B(k)$ is a martingale difference sequence, and from the Assumption~\ref{light_a1}, we have
$$\bbe[\exp\{(\gamma_k\langle \delta_k,x_k-x\rangle)^2/(2\gamma_k D_X M_F)^2\}]\leq \exp\{1\}.$$
Hence, we have
\begin{equation}\label{L10_3}
    \Prob\{\tsum_{k\in\B}\gamma_k\langle f'(x_k) - F'(x_k,\xi_k),x_k-x\rangle>2\lambda M_FD_X\sqrt{\tsum_{k=s}^N\gamma_k^2}\}\leq \exp\{-\tfrac{\lambda^2}{3}\}.
\end{equation}
Also, we have
\begin{equation}\label{L10_4}
    \Prob\{\tsum_{k\in\N}\gamma_k\langle g'(x_k) - G'(x_k,\xi_k),x_k-x\rangle>2\lambda M_G D_X\sqrt{\tsum_{k=s}^N\gamma_k^2}\}\leq \exp\{-\tfrac{\lambda^2}{3}\}.
\end{equation}
Besides,
\begin{align*}
& \Prob\{\tsum_{k\in\N} \gamma_kg(x_k) < \tsum_{k\in\N} \gamma_k\hat G_k -\tfrac{\lambda\sigma}{\sqrt{J}}\tsum_{k=s}^N\gamma_k\} \\
  \leq & \Prob\{\tsum_{k\in\N} \gamma_kg(x_k) < \tsum_{k\in\N} \gamma_k\hat G_k -\tfrac{\lambda\sigma}{\sqrt{J}}\tsum_{k\in\N}\gamma_k\}\\
  \leq & \Prob\{\exists k\in \N, g(x_k) < \hat G_k -\tfrac{\lambda\sigma}{\sqrt{J}}\}\leq 1-(1-\exp\{-\tfrac{\lambda^2}{3}\})^{|\N|}\leq (N-s)\exp\{-\tfrac{\lambda^2}{3}\},\\
\end{align*}
thus,
\begin{equation}\label{L10_5}
    \Prob\{\tsum_{k\in\N} \gamma_k[\hat G_k- g(x_k)]>  \tfrac{\lambda\sigma}{\sqrt{J}}\tsum_{k=s}^N\gamma_k\} \leq (N-s)\exp\{-\tfrac{\lambda^2}{3}\}
\end{equation}
Combining \eqref{L10_1}, \eqref{L10_2}, \eqref{L10_3}, \eqref{L10_4} and \eqref{L10_5}, we have \eqref{cond_prop2}.
Therefore, with probability at least $1-2\exp\{-\lambda\} - (N-s+2)\exp\{-\lambda^2/3\}$, we have
$$\tsum_{k\in\mathcal{N}}\gamma_k\eta_k+\tsum_{k\in \mathcal{B}}\gamma_k\langle f'(x_k),x_k-x^*\rangle \leq K_0+\lambda K_1.$$
If $\tsum_{k\in \B}\gamma_k\langle f'(x_k),x_k-x^*\rangle \leq 0$, part b) holds.
If $\tsum_{k\in \B}\gamma_k\langle f'(x_k),x_k-x^*\rangle \geq 0$, we have
\begin{equation}\label{5.4}
\tsum_{k\in \mathcal{N}}\gamma_k\eta_k\leq K_0+\lambda K_1.
\end{equation}
Suppose that $|\mathcal{B}| < (N-s+1)/2$, i.e., $| \mathcal{N}| \geq (N-s+1)/2$. Then, the condition in \eqref{5.3} implies that
$$\tsum_{k\in \mathcal{N}}\gamma_k\eta_k\geq \tfrac{N-s+1}{2}\min_{k\in \N}\gamma_k\eta_k >K_0 + \lambda K_0,$$
which yields contradiction. Then, part a) holds.
\end{proof}
%
%
Now we are ready to establish the large deviation properties of the CSA algorithm.
\begin{thm}\label{Theorem 5.1}
 Suppose that Assumptions~\ref{light_a} and~\ref{light_a1} hold. 
\begin{description}
  \item[a)] Suppose the parameters $s, \gamma_k, \eta_k$ are chosen by \eqref{5.3}, then for any given $ \lambda > 0$, we have
\begin{align}
    &\Prob\{f(\bar{x}_{N,s})-f(x^*)\leq \hat K_0 +\lambda\hat K_1\} \geq 1-2\exp\{-\lambda\}-(N-s+2)\exp\{-\tfrac{\lambda^2}{3}\},\label{5.5} \\
    &\Prob\left\{g(\bar{x}_{N,s})\leq \max_{k\in\B}\eta_k+\tfrac{\lambda\sigma}{\sqrt{J}}\right\}\geq 1- (N-s)\exp\{-\lambda^2/3\},\label{5.6}
\end{align}
where $\hat K_0=\left(\tfrac{N-s+1}{2}\min_{k\in \B}\gamma_k\right)^{-1}\left(D_X^2+(M_F^2+M_G^2)\tsum_{k =s}^N\gamma_k^2\right)$ and \\
$\hat K_1=\left(\tfrac{N-s+1}{2}\min_{k\in \B}\gamma_k\right)^{-1}\left((M_F^2+M_G^2)\tsum_{k =s}^N\gamma_k^2 +2(M_F+M_G)D_X\sqrt{\tsum_{k=s}^N\gamma_k^2}+\tfrac{\sigma}{\sqrt{J}}\tsum_{k=s}^N\gamma_k\right)$.
  \item[b)] For any $\Lambda\in (0,1)$ and $\lambda>1$, if we choose $N\leq \tfrac{\Lambda}{2}\exp\{\lambda^2/3\}$ and set
\begin{equation}\label{NJ_1}
\begin{aligned}
&s=1,\ \gamma_k=\tfrac{D_X}{\sqrt{N}M}, \ \eta_k=\tfrac{6MD_X}{\sqrt{N}}+\lambda(\tfrac{8MD_X}{\sqrt{N}}+\tfrac{2\sigma}{\sqrt{J}}),\\
&N = \max\{ \tfrac{24^2M^2D_X^2}{\epsilon^2}(\log\tfrac{4}{\Lambda})^2,  \tfrac{24^2M^2D_X^2}{\vartheta^2}(\log\tfrac{4}{\Lambda})^2\},\\
&J = \max\{ \tfrac{36\sigma^2}{\epsilon^2}(\log\tfrac{4}{\Lambda})^2,\tfrac{81\sigma^2}{\vartheta^2}(\log\tfrac{4}{\Lambda})^2 \},
\end{aligned}
\end{equation}
where $M = \max\{M_F,M_G\}$, then we have
\begin{equation} \label{lem_result_final}
\Prob\{g(\bar{x}_{N,s})\leq \vartheta\} \geq 1 -\Lambda  \text{ and }\  \Prob\{f(\bar{x}_{N,s})-f(x^*)\le \epsilon\}\geq 1-\Lambda.
\end{equation}

\end{description}
\end{thm}
\begin{proof}
Let us first show part a) holds.
From Lemma~\ref{Lemma 5.2}, for any given $\lambda>0$, with probability at least $1-2\exp\{-\lambda\}-(N-s+2)\exp\{-\tfrac{\lambda^2}{3}\} $, we have $\tsum_{k\in \B}\gamma_k\langle f'(x_k),x_k-x^*\rangle\leq K_0+\lambda K_1$, and either part a) or part b) in Lemma~\ref{Lemma 5.2} holds. If part b) holds, then from the convexity of $f$, we have
$f(\bar{x}_{N,s})-f(x^*)\leq 0$. If part a) holds, dividing both sides by $\tsum_{k\in \B}\gamma_k$, using the fact $|\B|\geq (N-s+1)/2$ and the convexity of $f$, we have $f(\bar{x}_{N,s})-f(x^*)\leq \hat K_0+\lambda\hat K_1$. Combining the above relations, we have \eqref{5.5}.
Let us show that \eqref{5.6} holds. Clearly, by the convexity of $g(\cdot)$ and the definition of $\bar{x}_{N,s}$, we have
$$g(\bar{x}_{N,s})=g(\tsum_{k\in \B}\iota_k x_k)\leq \left(\tsum_{k\in \B}\gamma_k\right)^{-1}\tsum_{k\in \B}\gamma_kg(x_k)\leq \max_{k\in\B}g(x_k).$$
Using this observation and the fact that $\Prob\{g(x_k)>\hat G_k + \tfrac{\lambda \sigma}{\sqrt{J}}\}<\exp\{-\tfrac{\lambda^2}{3}\}$, we obtain \eqref{5.6}.


Then, let us show part b) holds. First, observe that condition \eqref{5.3} holds by using the selection of $s$, $\{\gamma_k\}$ and $\{\eta_k\}$.
Thus from part a) and \eqref{NJ_1}, we have \eqref{lem_result_final}.
\end{proof}

In view of Theorem~\ref{Theorem 5.1}, the complexity in terms of the number of iterations $N$ of the CSA algorithm 
can be bounded by ${\cal O}(\max\{\tfrac{1}{\epsilon^2}(\log\tfrac{1}{\Lambda})^2,\tfrac{1}{\vartheta^2}\})$, 
and the sample size $J$ for estimating constraint in every iteration of the CSA algorithm
can be bounded by ${\cal O}(\max\{\tfrac{1}{\epsilon^2}(\log\tfrac{1}{\Lambda})^2,\tfrac{1}{\vartheta^2}\log\tfrac{1}{\Lambda^3}\})$ 
for solving problem~\eqref{1.1}-\eqref{1.1exp}.

\subsection{Strongly convex objective and strongly convex constraints}
In this subsection, we are interested in establishing the convergence of the CSA algorithm applied to strongly convex problems. More specifically, we assume that the objective function $F$ and constraint function $g$ in problem \eqref{1.1}, where $g$ is given in the form of function constraint, are both strongly convex w.r.t. $x$, i.e., $\exists \mu_F>0$ and $\mu_G>0$ s.t.
\begin{align*}
F(x_1,\xi)&\geq F(x_2,\xi)+\langle F'(x_2,\xi),x_1-x_2\rangle+\tfrac{\mu_F}{2}\|x_1-x_2\|^2, \forall x_1,x_2\in X,\\
g(x_1)&\geq g(x_2)+\langle g'(x_2),x_1-x_2\rangle+\tfrac{\mu_G}{2}\|x_1-x_2\|^2, \forall x_1,x_2\in X.
\end{align*}
For the sake of simplicity, we focus on the case when the constraint function $g$ can be evaluated exactly (i.e., $\hat G_k = g'(x_k)$). However, expectation constraints
can be dealt with using similar techniques
discussed in Section~\ref{sec_expconstr}.

In order to estimate the convergent rate of the CSA algorithm for solving strongly convex problems, we need to assume that the prox-function $V_X(\cdot,\cdot)$ and $V_Y(\cdot,\cdot)$ satisfies a quadratic growth condition
\begin{equation}\label{QuadraticCondition}
V_X(z,x)\leq \tfrac{Q}{2}\|z-x\|^2, \forall z,x \in X\text{ and} \ V_Y(z,y)\leq \tfrac{Q}{2}\|z-y\|^2, \forall z,y \in Y.
\end{equation}
{\color{black} Moreover, letting $\gamma_k$ be the stepsizes used in the CSA method, and denoting $$a_k = \left\{
            \begin{array}{ll}
              \tfrac{\mu_F\gamma_k}{Q}, & \hbox{$k\in \B$,} \\
              \tfrac{\mu_G\gamma_k}{Q}, & \hbox{$k\in \N$,}
            \end{array}
          \right.
A_k = \left\{
            \begin{array}{ll}
              1, & \hbox{$ k=1$,} \\
              (1-a_k)A_{k-1}, & \hbox{$k\geq2$,}
            \end{array}
          \right.
          \text{and }
          \rho_k = \tfrac{ \gamma_k}{A_k},
          $$
we define 
\begin{equation}\label{strong_output}
\bar{x}_{N,s} = \tfrac{\tsum_{k\in \B}\rho_k x_k}{\tsum_{k\in \B}\rho_k}
\end{equation}
as  the output of Algorithm 1.} 

The following simple result will be used in the convergence analysis of the CSA method.
\begin{lem}\label{Lemma 3}
If $a_k\in(0,1]$, k = 0,1,2,..., $A_k>0,\forall k\geq 1$, and $\{\Delta_k\}$ satisfies
$$\Delta_{k+1}\leq (1-a_k)\Delta_k + B_k, \forall k \geq 1,$$
then we have
$$\tfrac{\Delta_{k+1}}{A_k}\leq (1-a_1)\Delta_1 + \tsum_{i=1}^k\tfrac{B_i}{A_i}.$$
\end{lem}

Below we provide an important recursion about CSA applied to strongly convex problems. This result differs from Proposition \ref{Proposition 1} for the general convex case
in that we use different weight $\rho_k$ rather than $\gamma_k$.
\begin{prop}\label{Proposition 2}
For any $1\leq s\leq N$, we have
\begin{multline}\label{prop1.1}
    \tsum_{k\in \mathcal{N}}\rho_k(\eta_k-g(x))+\tsum_{k\in \B}\rho_k [F(x_k,\xi_k)-F(x,\xi_k)]\leq (1-a_s) D_X^2 \\
    +\tfrac{1}{2}\tsum_{k\in \mathcal{B}}\rho_k\gamma_k\|  F'(x_k,\xi_k)\|_*^2 +\tfrac{1}{2}\tsum_{k\in \mathcal{N}}\rho_k \gamma_k\| g'(x_k)\|_*^2.
\end{multline}
\end{prop}

\begin{proof}
Consider the iteration $k$, $\forall s\leq k\leq N$. If $k\in \B$, by Lemma \ref{Lemma 1} and the strong convexity of $F(x,\xi)$, we have
\[
\begin{aligned}
    V(x_{k+1},x) &\leq V(x_k,x)-\gamma_k\langle h_k,x_k-x\rangle+\tfrac{1}{2}\gamma_k^2\| F'(x_k,\xi_k)\|_*^2 \\
 &=  V(x_k,x)-\gamma_k\langle F'(x_k,\xi_k),x_k-x\rangle+\tfrac{1}{2}\gamma_k^2\|  F'(x_k,\xi_k)\|_*^2 \\
     &\leq V(x_k,x)-\gamma_k\left[F(x_k,\xi_k)-F(x,\xi_k)+\tfrac{\mu_F}{2}{\color{black}\|x_k-x\|^2}\right]+\tfrac{1}{2}\gamma_k^2\|  F'(x_k,\xi_k)\|_*^2 \\
     &\leq \left(1-\tfrac{\mu_F\gamma_k}{Q}\right)V(x_k,x) - \gamma_k[F(x_k,\xi_k)-F(x,\xi_k)] + \tfrac{1}{2}\gamma_k^2\| F'(x_k,\xi_k)\|_*^2.
\end{aligned}
\]
Similarly for $k\in\N$, using Lemma \ref{Lemma 1} and the strong convexity of $g(x)$, we have
\[
\begin{aligned}
    V(x_{k+1},x) &\leq V(x_k,x)-\gamma_k\langle h_k,x_k-x\rangle+\tfrac{1}{2}\gamma_k^2\| g'(x_k)\|_*^2 \\
 &=  V(x_k,x)-\gamma_k\langle g'(x_k),x_k-x\rangle+\tfrac{1}{2}\gamma_k^2\| g'(x_k)\|_*^2 \\
    & \leq V(x_k,x)-\gamma_k\left[(g(x_k)-g(x))+\tfrac{\mu_G}{2}{\color{black}\|x_k-x\|^2}\right]+\tfrac{1}{2}\gamma_k^2\| g'(x_k)\|_*^2 \\
     &\leq \left(1-\tfrac{\mu_G\gamma_k}{Q}\right)V(x_k,x) - \gamma_k(\eta_k-g(x))+ \tfrac{1}{2}\gamma_k^2\| g'(x_k)\|_*^2.
\end{aligned}
\]
Summing up these inequalities for $s\leq k\leq N$ and using Lemma \ref{Lemma 3}, we have
\[
\begin{aligned}
    \tfrac{V(x_{N+1},x)}{A_N} &\leq \left(1-a_s\right)V(x_s,x) - \left[\tsum_{k\in \mathcal{N}}\tfrac{\gamma_k}{A_k}(\eta_k-g(x))+ \tsum_{k\in \mathcal{B}}\tfrac{\gamma_k}{A_k}[F(x_k,\xi_k)-F(x,\xi_k)]\right] \\
    &+\tfrac{1}{2}\tsum_{k\in \mathcal{N}}\tfrac{\gamma_k^2}{A_k}\| g'(x_k)\|_*^2 +\tfrac{1}{2}\tsum_{k \in \mathcal{B}}\tfrac{\gamma_k^2}{A_k}\| F'(x_k,\xi_k)\|_*^2,
\end{aligned}
\]
Using the fact $V(x_{N+1},x)/A_N\geq 0$ and  the definition of $\rho_k$, and rearranging the terms in the above inequality, we obtain \eqref{prop1.1}.
\end{proof}

\vgap

Lemma \ref{Lemma 4} below provides a sufficient condition which guarantees $\bar{x}_{N,s}$ to be well-defined.
\begin{lem}\label{Lemma 4}
Let $x^*$ be the optimal solution of \eqref{1.1}, then for some $\rho\in (0,1)$, with probability at least $1-\rho$, we have
\begin{equation}\label{cond_prop3}
\tsum_{k\in \mathcal{N}}\rho_k\eta_k+\tsum_{k\in \B}\rho_k [f(x_k)-f(x^*)]\leq \tfrac{1}{\rho}(1-a_s)D_X^2+ \tfrac{M^2}{2\rho}\tsum_{k=s}^N\rho_k\gamma_k.
\end{equation}
 If we have
\begin{equation}\label{2.5}
  \tfrac{N-s+1}{2} \min_{k\in \N}\rho_k\eta_k >\tfrac{1}{\rho}(1-a_s)D_X^2+ \tfrac{M^2}{2\rho}\tsum_{k=s}^N\rho_k\gamma_k,
\end{equation}
then conditional on \eqref{cond_prop3}, $\B\neq \emptyset$ and hence $\bar{x}_{N,s}$ is well-defined. Moreover, we have one of the following two statements holds,
\begin{description}
  \item[a)] $|\mathcal{B}| \geq (N-s+1)/2,$
  \item[b)] $
  \tsum_{k\in \B}\rho_k[ f(x_k)-f(x^*)]\leq 0.
 $
\end{description}
\end{lem}
\begin{proof}
The proof of this result is similar to that of Lemma 2 and hence the details are skipped.
\end{proof}

\vgap

With the help of Proposition \ref{Proposition 2}, we are ready to establish the main convergence properties of the CSA method for solving strongly convex problems.
\begin{thm}\label{Theorem 2}
Suppose that $\{\gamma_k\}$ and $\{\eta_k\}$ in the CSA algorithm are chosen such that \eqref{2.5} holds.
Then, with probability $1-\rho$, for any $1\leq s\leq N$, we have
\begin{align}
 & f(\bar{x}_{N,s})-f(x^*) \leq  \tfrac{2(1-a_s)D_X^2+M^2\tsum_{k=s}^N\rho_k\gamma_k}{\rho(N-s+1)\min_{s\leq k\leq N}\rho_k} , \label{CSCResult1}\\
  & g(\bar{x}_{N,s}) \leq (\tsum_{k\in \B}\rho_k)^{-1}(\tsum_{k\in \B}\rho_k\eta_k),\label{CSCResult2}
\end{align}
where $M = \max\{M_F,M_G\}$.
\end{thm}
\begin{proof}
The proof of this theorem is similar to the one of Theorem \ref{Theorem 1} and hence the details are skipped.
\end{proof}

Below we provide a stepsize policy of $s$, $\gamma_k$ and $\eta_k$ in order to achieve the optimal rate of convergence for solving strongly convex problems.
\begin{cor}\label{Coro 3}
Let $s=\tfrac{N}{2}$, $\gamma_k = \left\{
                  \begin{array}{ll}
                    \tfrac{2Q}{\mu_F(k+1)} , & \hbox{if $k\in \B$;} \\
                    \tfrac{2Q}{\mu_G(k+1)} , & \hbox{if $k\in \N$,}
                  \end{array}
                \right.$, $\eta_k = \tfrac{2\mu_G Q}{k \rho}\left(\tfrac{2D_X^2}{k}+ \tfrac{M^2}{\mu^2}\right)$, then with probability $1-\rho$, we have
\[
\begin{aligned}
   &f(\bar{x}_{N,s})-f(x^*) \leq  \tfrac{4\mu_FD_X^2}{N^2Q\rho} +  \tfrac{2\mu_FM^2Q }{N\mu^2\rho},\\
   &g(\bar{x}_{N,s})\leq \tfrac{8\mu_G Q D_X^2}{N^2\rho}+\tfrac{4\mu_G M^2Q}{N\mu^2\rho},
\end{aligned}
\]
where $M =\max \{M_F,M_G\}$ and $\mu = \min\{\mu_F,\mu_G\}$.
\end{cor}
\begin{proof}
Based on our selection of $s$, $\gamma_k$, $\eta_k$ and the definition of $a_k$, $A_k$ and $\rho_k$, we have
$$a_k = \tfrac{2}{k+1},\ A_k = \prod_{i=2}^k (1-a_i) = \tfrac{2}{k(k+1)},\ \rho_k=\left\{
                  \begin{array}{ll}
                    \tfrac{kQ}{\mu_F} , & \hbox{if $k\in \B$;} \\
                    \tfrac{kQ}{\mu_G} , & \hbox{if $k\in \N$,}
                  \end{array}
                \right.$$
For $\forall s\leq k\leq N$, by the definition of $s$, $\gamma_k$ and $\eta_k$, we have 
\begin{multline*}
   \tfrac{1}{\rho} (1-a_s)V(x_s,x)+ \tfrac{1}{2\rho}\tsum_{k=s}^N\rho_k\gamma_k M^2 \\
    \leq \tfrac{1}{\rho}D_X^2 +\tfrac{1}{2\rho}\tsum_{k\in \B}\tfrac{\gamma_k^2}{A_k}M^2+\tfrac{1}{2\rho}\tsum_{k\in \N}\tfrac{\gamma_k^2}{A_k}M^2 \leq \tfrac{D_X^2}{\rho}+ Q^2 (|\B|\tfrac{M^2}{\mu_F^2}+|\N|\tfrac{M^2}{\mu_G^2})\leq \tfrac{D_X^2}{\rho}+ \tfrac{M^2Q^2 N}{2\mu^2\rho},
\end{multline*}
\begin{equation*}
  \tfrac{N-s+1}{2} \min_{k\in \N}\rho_k\eta_k =\tfrac{N}{4}\min_{k\in \N}\tfrac{kQ}{\mu_G}\tfrac{2\mu_G Q}{k\rho}\left(\tfrac{2D_X^2}{k}+ \tfrac{M^2}{\mu^2}\right)
=  \tfrac{D_X^2}{\rho} +  \tfrac{M^2Q^2 N}{2\mu^2\rho}.
\end{equation*}
Combining the above two inequalities, we can easily see that condition \eqref{2.5} holds.
It then follows from Theorem \ref{Theorem 2} that with probability $1-\rho$ 
\begin{align*}
f(\bar{x}_{N,s})-f(x^*) \leq &(\rho (N-s+1)\min_{s\leq k\leq N}\rho_k)^{-1}\left(2(1-a_s)D_X^2+\tsum_{k=s}^N\rho_k\gamma_k M^2\right)\\
\leq& \tfrac{4\mu_FD_X^2}{N^2Q\rho} +  \tfrac{2\mu_FM^2Q }{N\mu^2\rho},\\
g(\bar{x}_{N,s})\leq &(\tsum_{k\in \B}\rho_k)^{-1}(\tsum_{k\in \B}\rho_k\eta_k)\leq \tfrac{8\mu_G Q D_X^2}{N^2\rho}+\tfrac{4\mu_G M^2Q}{N\mu^2\rho}.
\end{align*}
\end{proof}

In view of Corollary \ref{Coro 3}, the CSA algorithm can achieve the optimal rate of convergence for strongly convex optimization with strongly convex constraints.
To the best of our knowledge, this is the first time such a complexity result is obtained in the literature and this result is new also for the deterministic setting.

\setcounter{equation}{0}
\section{Expectation constraints over problem parameters}
In this section, we are interested in solving a class of parameterized stochastic optimization problems whose parameters
are defined by expectation constraints as described in \eqref{para1}-\eqref{para2}, under the assumption that such a pair of solutions
satisfying \eqref{para1}-\eqref{para2} exists.

Our goal in this section is to present a variant of the CSA algorithm to approximately solve problem~\eqref{para1}-\eqref{para2} and
establish its convergence properties. More specifically, we discuss this variant of the CSA algorithm when applied to
the parameterized stochastic optimization problem in \eqref{para1}-\eqref{para2}  and then consider a modified problem
by imposing certain strong convexity assumptions to the function $\Phi(x,y,\zeta)$ w.r.t. $y$ and $G(x,\xi)$ w.r.t. $x$ in Subsections 4.1 and 4.2,
respectively.
In Subsection 4.3, we discuss some large deviation properties for the variant of the CSA method for the problem defined by \eqref{para1}-\eqref{para2}.

\subsection{Stochastic optimization with parameter feasibility constraints}

Given tolerance $\eta >0$ and target accuracy $\epsilon > 0$, we will present a variant of the CSA algorithm,
namely cooperative stochastic parameter approximation (CSPA),  to find a pair of approximate solutions $(\bar{x},\bar{y}) \in X \times Y$
s.t. $\bbe[g(\bar{x})] \leq \eta$ and
$
\bbe[ \phi(\bar{x},\bar y) - \phi(\bar x,y)] \leq \epsilon, \ \forall y\in Y,
$ in this subsection. Before we describe the CSPA method, we need slightly modify Assumption 1.
\begin{assumption}\label{CSPA_assump1}
For any $x\in X$ and $y\in Y$,
$$\bbe[\| \Phi'(x,y,\zeta)\|_*^2]\leq M_\Phi^2 \ \ \ \mbox{and} \ \ \ \bbe[\|G'(x,\xi)\|_*^2]\leq M_G^2,$$
where $\Phi'(x,y,\zeta) \in \partial_y \Phi(x, y,\zeta)$ and $G'(x,\xi) \in \partial_x G(x, \xi)$.
\end{assumption}
We will also discuss the convergent properties under the light-tail assumptions as follows.
\begin{assumption}\label{CSPA_lighttail}
\begin{align*}
\bbe[\exp\{\| \Phi'(x,y,\zeta)\|_*^2/M_\Phi^2\}]\leq\exp\{1\},\\
\bbe[\exp\{(\Phi(x,y,\zeta)-\phi(x,y))^2/\sigma^2\}]\leq \exp\{1\},\\
\bbe[\exp\{(G(x,\xi)-g(x))^2/\sigma^2\}]\leq \exp\{1\}.
\end{align*}
\end{assumption}
We assume that the distance generating functions $\omega_X : X \mapsto \bbr$ and $\omega_Y : Y \mapsto \bbr$ are strongly convex with modulus $1$ w.r.t.
given norms in $\bbr^n$ and $\bbr^m$, respectively, and that their associated prox-mappings $P_{x,X}$ and $P_{y,Y}$ (see \eqref{Prox-mapping def}) are easily computable.

We make the following modifications to the CSA method in Section 2.1 in order to apply it to solve problem \eqref{para1}-\eqref{para2}. {\color{black} Firstly, we still check the solution $(x_k,y_k)$ to see whether $x_k$ violates the condition $\tsum_{i=1}^k\gamma_iG(x_i,\xi_i)/\tsum_{i=1}^k\gamma_i\leq \eta_k$. If so, we set the search direction as $G'(x_k,\xi_k)$ to update $x_k$, while keeping $y_k$ intact. Otherwise, we only update $y_k$ along the direction $\Phi'(\bar x_k,y_k,\zeta_k)$. Secondly, we define the output as a randomly selected $(\bar x_k,y_k)$ according to a certain probability distribution instead of the ergodic mean of $\{(\bar x_k,y_k)\}$, where $ \bar x_k$ denotes
the average of $\{x_k\}$ (see \eqref{def_avg_para}).
Since we are solving a coupled optimization and feasibility problem, each iteration of our algorithm only updates either $y_k$ or $x_k$ and requires the computation
of either  $\Phi'$ or $G'$  depending on whether $\tsum_{i=1}^k\gamma_iG(x_i,\xi_i)/\tsum_{i=1}^k\gamma_i\leq\eta_k$.} This differs from the SA method used in Jiang and Shanbhag \cite{jiang2014solution}
that requires two projection steps and the computation of two subgradients at each iteration to solve a different parameterized stochastic optimization problem.

\begin{algorithm}
\caption{The cooperative stochastic parameter approximation method}
\ \ \ {\bf Input:} initial point ${\color{black}(x_1,y_1)}$, stepsize $\{\gamma_k\}$, tolerance $\{\eta_k\}$, number of iterations $N$, $\tau(1)=1$.
\begin{algorithmic}
\State
{\bf for } k=1,2,...,N \\
  \indent\ \ \ \ \ \  {\bf if} {\color{black} $\tsum_{i=1}^{\tau(k)}\gamma_iG(x_i,\xi_i)/\tsum_{i=1}^{\tau(k)}\gamma_i\leq\eta_k$}
  \indent\begin{equation} \label{def_avg_para}
      {\color{black} y_{k+1}=P_{y_k, Y}(\gamma_k \Phi'(\bar x_k,y_k,\zeta_k)), \ \tau(k+1)=\tau(k), \text{ where } \bar x_k = \tsum_{i=1}^{\tau(k)}\gamma_ix_i/\tsum_{i=1}^{\tau(k)}\gamma_i};
  \indent\end{equation}
  \indent{\bf else}
  \begin{equation}
    {\color{black} l=\tau(k), x_{l+1}=P_{x_l,X} (\gamma_l G'(x_l,\xi_l)), y_{k+1}=y_k,\ \tau(k+1)=\tau(k)+1.}
  \end{equation}
\indent{\bf end if}

  \State {\bf end for}
  \State {\bf Output:} Set {\color{black} $\B := \{s\leq k\leq N| \tsum_{i=1}^{\tau(k)}\gamma_iG(x_i,\xi_i)/\tsum_{i=1}^{\tau(k)}\gamma_i\leq \eta_k\}$} for some $ 1 \le s \le N$, and define the output $(\bar x_R,y_R)$, where R is randomly chosen according to
       \begin{equation}\label{PR}
       \Prob\{R=k\}=\tfrac{\gamma_k}{\tsum_{k\in \mathcal{B}}\gamma_k},k\in \mathcal{B}.
       \end{equation}

\end{algorithmic}
\end{algorithm}

With a little abuse of notation, we still use $\B$ to represent the set $\{s\leq k\leq N| {\color{black} \tsum_{i=1}^{\tau(k)}\gamma_iG(x_i,\xi_i)/\tsum_{i=1}^{\tau(k)}\gamma_i\leq \eta_k}\}$, $I = \{s, \ldots, N\}$, and $\N = I \setminus \B$. The following result mimics Proposition \ref{Proposition 1}.
\begin{prop}\label{Proposition 3}
For any $1\leq s\leq N$, we have
\begin{align}
\tsum_{k\in \B}\gamma_k\langle\Phi'(\bar x_k,y_k,\zeta_k),y_k-y\rangle
& \leq  D_Y^2+ \tfrac{1}{2}\tsum_{k\in \mathcal{B}}\gamma_k^2\| \Phi'(\bar x_k,y_k,\zeta_k)\|_*^2,\ \forall y\in Y, \label{3.2.1} \\
\tsum_{i=\tau(s)}^{\tau(N)}\gamma_i[G(x_i,\xi_i)-G(x,\xi_i)] & \leq D_X^2+\tfrac{1}{2}\tsum_{i=\tau(s)}^{\tau(N)}\gamma_i^2\| G'(x_i,\xi_i)\|_*^2,\ \forall x\in X, \label{3.2.2}
\end{align}
where $D_X \equiv D_{X, w_x}$ and $D_Y \equiv D_{Y, w_y}$ are defined as in \eqref{D_X_def}.
\end{prop}

\begin{proof}
By Lemma \ref{Lemma 1}, if $k\in \mathcal{B}$,
$$
  V(y_{k+1},y)\leq V(y_k,y)+ \gamma_k\langle \Phi'(\bar x_k,y_k,\zeta_k), y-y_k \rangle + \tfrac{1}{2}\gamma_k^2\|\Phi'(\bar x_k,y_k,\zeta_k)\|_*^2.
$$
Also note that $V(y_{k+1},y) = V(y_k,y)$ for $k\in \N$. Summing up these relations for $k\in \B\cup\N$ and using the fact that $V(y_s,y) \leq D_Y^2$, we have
\begin{equation}\label{3.2.3}
\begin{aligned}
 V(y_{N+1},y) & \leq V(y_s,y)+\tfrac{1}{2}\tsum_{k\in \mathcal{B}}\gamma_k^2\|\Phi'(\bar x_k,y_k,\zeta_k)\|_*^2  -\tsum_{k\in \mathcal{B}}\gamma_k\langle \Phi'(\bar x_k,y_k,\zeta_k),y_k-y\rangle\\
   &\leq D_Y^2+\tfrac{1}{2}\tsum_{k\in \mathcal{B}}\gamma_k^2\|\Phi'(\bar x_k,y_k,\zeta_k)\|_*^2 -\tsum_{k\in \mathcal{B}}\gamma_k\langle \Phi'(\bar x_k,y_k,\zeta_k),y_k-y\rangle.
\end{aligned}
\end{equation}
Similarly for $\tau(s)\leq i \leq \tau(N)$, we have
$$
\begin{aligned}
  V(x_{i+1},x) & \leq V(x_i,x) + \gamma_i\langle G'(x_i,\xi_i), x-x_i \rangle + \tfrac{1}{2}\gamma_i^2\|G'(x_i,\xi_i)\|_*^2.
\end{aligned}
$$
Summing up these relations for $\tau(s)\leq i \leq \tau(N)$ and using the fact that $V(x_{\tau(s)},x)\leq D_X^2$, we obtain
\begin{equation}\label{3.2.4}
\begin{aligned}
V(x_{\tau(N)+1},x) \leq D_X^2+\tsum_{i=\tau(s)}^{\tau(N)}\gamma_i^2\|G'(x_i,\xi_i)\|_*^2 -\tsum_{i=\tau(s)}^{\tau(N)}(G(x_i,\xi_i)-G(x,\xi_i)).
\end{aligned}
\end{equation}
Using the facts $V(y_{N+1},y)\geq 0$ and $V(x_{\tau(N)+1},x)\geq 0$, and rearranging the terms in \eqref{3.2.3} and \eqref{3.2.4}, we then obtain \eqref{3.2.1} and \eqref{3.2.2}, respectively.
\end{proof}

\vgap

The following result provides a sufficient condition under which $(\bar x_R,y_R)$ is well-defined.
\begin{lem}\label{Lemma 6}
{\color{black} The following statements holds.

\begin{description}
  \item[a)] Under Assumption~\ref{CSPA_assump1}, if for any $\lambda>0$, we have
\begin{equation}\label{3.3}
    \tfrac{N-s+1}{2}\min_{k\in \N}{\gamma_k\eta_k} > \lambda D_X^2+\lambda\tfrac{M_G^2}{2} \tsum_{k=s}^{N}\gamma_k^2,
\end{equation}
then $\Prob\{|\mathcal{B}|\geq \tfrac{N-s+1}{2}\} \geq 1- 1/\lambda$.
  \item[b)] Under Assumption~\ref{CSPA_lighttail}, if for any $\lambda>0$, we have
\begin{equation}\label{3.4}
    \tfrac{N-s+1}{2}\min_{k\in \N}{\gamma_k\eta_k} > D_X^2+(1+\lambda)\tfrac{M_G^2}{2}\tsum_{k=\tau(s)}^{\tau(N)}\gamma_k^2+\lambda\sigma\sqrt{ \tsum_{k=s}^{N}\gamma_k^2},
\end{equation}
then $\Prob\{|\mathcal{B}|\geq \tfrac{N-s+1}{2}\} \geq 1- \exp\{-\lambda\}-\exp\{-\tfrac{\lambda^2}{3}\}$.
\end{description}}
\end{lem}
\begin{proof}
First let us show part a), set $\delta_k = G(x^*,\xi_k) - g(x^*)$, it follows from \eqref{3.2.2} with $x = x^*$ that
$$\tsum_{i=\tau(s)}^{\tau(N)}\gamma_iG(x_i,\xi_i)-\tsum_{i=\tau(s)}^{\tau(N)}\gamma_ig(x^*) \leq D_X^2+\tfrac{1}{2}\tsum_{i=\tau(s)}^{\tau(N)}\gamma_i^2\| G'(x_i,\xi_i)\|_*^2+\tsum_{i=\tau(s)}^{\tau(N)}\gamma_i\delta_i.$$
The above relation, in view of $g(x^*)\leq 0$ and the fact $\tsum_{i=\tau(s)}^{\tau(N)}\gamma_iG(x_i,\xi_i)\geq \eta_{\tau(N)}\tsum_{i=\tau(s)}^{\tau(N)}\gamma_i$, implies that
$$ |\N|\min_{k\in \N}{\gamma_k\eta_k}\leq \eta_{\tau(N)}\tsum_{k=\tau(s)}^{\tau(N)}\gamma_k \leq D_X^2+\tfrac{1}{2}\tsum_{k=\tau(s)}^{\tau(N)}\gamma_k^2\| G'(x_k,\xi_k)\|_*^2+\tsum_{k=\tau(s)}^{\tau(N)}\gamma_k\delta_k.
$$
Under Assumption~\ref{CSPA_assump1},  for any $\lambda>0$, using Markov's inequality, we have
$$\Prob\{|\N|\min_{k\in \N}{\gamma_k\eta_k} \leq \lambda D_X^2+\lambda\tfrac{M_G^2}{2} \tsum_{k=s}^{N}\gamma_k^2\} \geq 1-1/\lambda.$$
For contradiction, suppose that $|\mathcal{B}|< \tfrac{N-s+1}{2}$, i.e., $\tau(N)-\tau(s)=|\N|\geq \tfrac{N-s+1}{2}$. Combining \eqref{3.3}, we obtain the contradiction. Hence, part a) holds.
Under Assumption~\ref{CSPA_lighttail}, for any $\lambda>0$, it follows from the Markov's inequality and Bernstein's inequality that
$$
   \Prob\left(\tsum_{k=\tau(s)}^{\tau(N)}\gamma_k^2\| G'(x_k,\xi_k)\|_*^2> (1+\lambda)M_G^2\tsum_{k=\tau(s)}^{\tau(N)}\gamma_k^2\right)\leq \exp\{-\lambda\}
  $$
 and
$$
    \Prob\left\{\tsum_{k=\tau(s)}^{\tau(N)}\gamma_k\delta_k>\lambda\sigma \sqrt{\tsum_{k=s}^N\gamma_k^2}\right\}\leq \exp\{-\lambda^2/3\}.
$$
Hence,
$$\Prob\left\{ |\N|\min_{k\in \N}{\gamma_k\eta_k} \leq D_X^2+(1+\lambda)\tfrac{M_G^2}{2}\tsum_{k=\tau(s)}^{\tau(N)}\gamma_k^2+\lambda\sigma\sqrt{ \tsum_{k=s}^{N}\gamma_k^2} \right\} \geq 1-  \exp\{-\lambda\}-\exp\{-\tfrac{\lambda^2}{3}\}.$$
For contradiction, suppose that $|\mathcal{B}|< \tfrac{N-s+1}{2}$, i.e., $\tau(N)-\tau(s)=|\N|\geq \tfrac{N-s+1}{2}$. 
Combining \eqref{3.4}, part b) holds. 
\end{proof}

\vgap

Theorem \ref{Theorem 4} summarizes the main convergence properties of Algorithm 2 applied to problem \eqref{para1}-\eqref{para2}.
\begin{thm}\label{Theorem 4}
The following statements holds for the CSPA algorithm. 
\begin{description}
  \item[a)] Under Assumption~\ref{CSPA_assump1}, we have, $\forall \lambda >0$,
\begin{align}
  &\Prob\{\bbe_R[\phi(\bar x_R,y_R)-\phi(\bar x_R,y^*(\bar x_R))] \leq \lambda \tfrac{2D_Y^2+M_\Phi^2\tsum_{k\in \mathcal{B}}\gamma_k^2}{2\tsum_{k\in \B}\gamma_k}\}\geq 1-1/\lambda,\label{3.3.2}\\
 &\Prob\left\{g(\bar x_R) \leq  \eta_R+\sigma\lambda\tfrac{\sqrt{\tsum_{k=\tau(s)}^{\tau(N)}\gamma_k^2}}{\tsum_{k=\tau(s)}^{\tau(N)}\gamma_k}\right\}\geq 1- \tfrac{1}{\lambda^2}.\label{3.3.3}
\end{align}
  \item[b)] Under Assumption~\ref{CSPA_lighttail}, we have,  $\forall \lambda>0$,
  \begin{align}
 &\Prob\left\{ \bbe_R[\phi(\bar x_R,y_R)-\phi(\bar x_R,y^*(\bar x_R)) ]\geq K_0+\lambda K_1\right\}\leq \exp\{-\lambda\}+\exp\{-\lambda^2/3\},\label{3.4.2}\\
 &\Prob\left\{g(\bar x_R) \geq \eta_R+\lambda\sigma\tfrac{\sqrt{\tsum_{k=\tau(s)}^{\tau(N)}\gamma_k^2}}{\tsum_{k=\tau(s)}^{\tau(N)}\gamma_k}\right\}\leq \exp\{-\lambda^2/3\},\label{3.4.3}
\end{align}
where $K_0 = \tfrac{2D_Y^2+M_\Phi^2\tsum_{k\in \mathcal{B}}\gamma_k^2}{2\tsum_{k\in \B}\gamma_k}$ and $K_1 = \tfrac{M_\Phi^2\tsum_{k\in \B}\gamma_k^2+4M_\Phi D_Y\sqrt{\tsum_{k\in \B}\gamma_k^2}}{2\tsum_{k\in \B}\gamma_k}$.
\end{description}
\end{thm}
\begin{proof}
Let us prove part a) first. Set $\Delta_k = \Phi(\bar x_k,y_k,\zeta_k) - \phi(\bar x_k, y_k)$, it follows from \eqref{3.2.1} (fix $y = y^*$) that
\begin{equation}\label{3.5}
\tsum_{k\in  \mathcal{B}}\gamma_k\left[ \phi(x_k,y_k)-\phi(x_k,y^*(x_k))\right] \leq D_Y^2 +\tfrac{1}{2}\tsum_{k\in \mathcal{B}}\gamma_k^2\|\Phi'(\bar x_k,y_k,\zeta_k)\|_*^2+\tsum_{k\in  \mathcal{B}}\gamma_k\Delta_k(y-y_k).
\end{equation}
Since conditional on $\zeta_{[k-1]}$, the expectation of $\Delta_k$ equals to zero, then taking expectation on both sides of \eqref{3.5}, we have
$$\bbe[\tsum_{k\in  \mathcal{B}}\gamma_k\left[ \phi(x_k,y_k)-\phi(x_k,y^*(x_k))\right]] \leq D_Y^2 +\tfrac{M_\Phi^2}{2}\tsum_{k=s}^N\gamma_k^2.$$
Using the Markov's inequality, we have
$$\Prob\{\tsum_{k\in  \mathcal{B}}\gamma_k\left[ \phi(x_k,y_k)-\phi(x_k,y^*(x_k))\right]\leq \lambda(D_Y^2 +\tfrac{M_\Phi^2}{2}\tsum_{k=s}^N\gamma_k^2)\} \geq 1-1/\lambda.$$
Hence, dividing both sides by $\tsum_{k\in\B}\gamma_k$ and definition of $R$, we have \eqref{3.3.2}.
Denote $\delta_k = G(x_k,\xi_k) - g(x_k)$. It then follows from the convexity of $g(\cdot)$ and the definition of the set $ \mathcal{B}$ that
\begin{equation}\label{3.6}
g(\bar x_k) \leq \tfrac{\tsum_{k=\tau(s)}^{\tau(N)}\gamma_kg(x_k)}{\tsum_{k=\tau(s)}^{\tau(N)}\gamma_k} \leq \eta_k - \tfrac{\tsum_{k=\tau(s)}^{\tau(N)}\gamma_k\delta_k}{\tsum_{k=\tau(s)}^{\tau(N)}\gamma_k}.
\end{equation}
Using the fact that $\bbe [\delta_k|\xi_{[k-1]}]=0$ and $\bbe[|\delta_k|^2]\leq \sigma^2$, we have
$$\bbe\left[ \left|\tfrac{\tsum_{k=\tau(s)}^{\tau(N)}\gamma_k\delta_k}{\tsum_{k=\tau(s)}^{\tau(N)}\gamma_k}\right|^2 \right] \leq \tfrac{\tsum_{k=\tau(s)}^{\tau(N)}\gamma_k^2\sigma^2}{(\tsum_{k=\tau(s)}^{\tau(N)}\gamma_k)^2}.$$
From the Chebyshev's inequality, we have \eqref{3.3.3}. Hence the part a) holds.

Under Assumption~\ref{CSPA_lighttail}, \eqref{3.5} still holds.
Using the fact that $\bbe[\exp\{\| \Phi'(\bar x_k,y_k,\zeta_k)\|_*^2/M_\Phi^2\}]\leq \exp\{1\}$ and Jensen's inequality, we have
$\bbe[\exp\{\tsum_{k\in \B}\gamma_k^2\| \Phi'(\bar x_k, y_k,\zeta_k)\|_*^2/M_\Phi^2\tsum_{k\in\B}\gamma_k^2\}]\leq \exp\{1\}.$
It then follows from Markov's inequality that $\forall \lambda\geq 0$,
\begin{equation}\label{3.5.1}
\begin{aligned}
   &\Prob(\tsum_{k\in \B}\gamma_k^2\| \Phi'(\bar x_k, y_k,\zeta_k)\|_*^2> (1+\lambda)M_\Phi^2\tsum_{k\in\B}\gamma_k^2) 
   \leq \tfrac{\exp\{1\}}{\exp\{1+\lambda\}} \leq \exp\{-\lambda\}.
\end{aligned}
\end{equation}
Also,
\begin{equation}\label{3.5.2}
\Prob\{\tsum_{k\in \B}\gamma_k\Delta_k(y-y_k)> 2\lambda M_\Phi D_Y\sqrt{\tsum_{k\in\B}\gamma_k^2}\}\leq \exp\{-\lambda^2/3\}
\end{equation}
Combining \eqref{3.5}, \eqref{3.5.1} and \eqref{3.5.2}, we have \eqref{3.4.2}.
Similarly, we have
\begin{equation}\label{3.6.1}
\Prob\{\tsum_{k=\tau(s)}^{\tau(N)}\gamma_k\delta_k\geq \lambda \sigma\sqrt{\tsum_{k=\tau(s)}^{\tau(N)}\gamma_k^2}\}\leq \exp\{-\lambda^2/3\}
\end{equation}
Combining \eqref{3.6} and \eqref{3.6.1}, we have \eqref{3.4.3}.
\end{proof}

\vgap

Below we provide a special selection of $s$, $\{\gamma_k\}$ and $\{\eta_k\}$.
\begin{cor}\label{Coro 4}
Denote $\nu := (M_GD_Y)/(M_\Phi D_X).$ Then we have the following statements hold.
\begin{description}
  \item[a)] Under Assumption~\ref{CSPA_assump1}, if $s=\tfrac{N}{2}+1$, $\gamma_k= \tfrac{D_X}{M_G\sqrt{k}}$ and $\eta_k=\tfrac{6M_GD_X}{\sqrt{k}\rho}$ for $k = 1,\ldots, N$, then 
\begin{align}
 &\Prob\left\{\bbe_R[\phi(\bar x_R,y_R)-\phi(\bar x_R,y^*(\bar x_R))] \leq \lambda\tfrac{8M_\Phi D_Y}{\sqrt{N}}\max\{ \nu, \tfrac{1}{\nu}\}\right\}\geq (1-\tfrac{1}{\lambda})(1-\rho),\label{3.7.1}\\
 &\Prob\left\{g(\bar x_R) \leq \lambda\tfrac{\sqrt{2}D_X}{\rho M_G\sqrt{N}}\right\}\geq 1- \tfrac{1}{\lambda^2}.\label{3.7.2}
\end{align}
  \item[b)] Under Assumption~\ref{CSPA_lighttail},  if $s=\tfrac{N}{2}+1$, $\gamma_k= \tfrac{D_X}{M_G\sqrt{k}}$ and $\eta_k=\tfrac{M_G}{\sqrt{k}}(6D_X+\tfrac{2}{\rho}D_X+\tfrac{4\sigma}{\rho})$ for $k = 1,\ldots, N$, then 
  \begin{align*}
 &\Prob\left\{\bbe_R[\phi(\bar x_R,y_R)-\phi(\bar x_R,y^*(\bar x_R))] \leq K_0+\lambda K_1\right\} \\
 & \geq ( 1-\exp\{-\lambda\}-\exp\{-\lambda^2/3\})(1-\exp\{- 1/\rho \}-\exp\{-1/3\rho^2\}),\\ 
 &\Prob\left\{g(\bar x_R) \leq\tfrac{\sqrt{2}D_X}{\rho M_G\sqrt{N}}+\lambda\tfrac{5\sigma}{\sqrt{N}}\right\}\geq 1-\exp\{-\lambda^2/3\},
\end{align*}
where $K_0 = \tfrac{8M_\Phi D_Y}{\sqrt{N}}\max\{ \nu, \tfrac{1}{\nu}\}$ and $K_1 = \tfrac{1}{\sqrt{N}}\left(\tfrac{4M_\Phi^2D_X}{M_G}+10M_\Phi D_Y\right)$.
\end{description}
\end{cor}
\begin{proof}
Similarly to Corollary \ref{Coro 1}, we can show that 
$$\tfrac{N-s+1}{2}\min_{k\in \N}{\gamma_k\eta_k} = \tfrac{N}{4}\min_k \tfrac{6D_X^2}{k\rho} = \tfrac{3D_X^2}{2\rho},$$
$$ \tfrac{ D_X^2}{\rho}+\tfrac{M_G^2}{2\rho}\tsum_{k=s}^{N}\gamma_k^2 =  \tfrac{ D_X^2}{\rho}+\tfrac{M_G^2}{2\rho}\tsum_{k=s}^{N}\tfrac{D_X^2}{M_G^2 k} <  \tfrac{ D_X^2}{\rho}+ \tfrac{ D_X^2\log2}{2\rho}.$$
Hence by Lemma \ref{Lemma 6}.a), we have $\Prob\{|\B| \geq \tfrac{N}{4}\}\geq 1-\rho$.
It then follows from Theorem \ref{Theorem 4} a) that
$$\tsum_{k\in \B}\gamma_k = \tsum_{k\in \B}\tfrac{D_X}{M_G\sqrt{k}}\geq \tfrac{D_X}{M_G}\tfrac{N}{4}\tfrac{1}{\sqrt{N}} = \tfrac{D_X\sqrt{N}}{4M_G}.$$
$$
\begin{aligned}
\tfrac{2D_Y^2+M_\Phi^2\tsum_{k\in \mathcal{B}}\gamma_k^2}{2\tsum_{k\in \B}\gamma_k}&\leq \tfrac{2M_G }{D_X\sqrt{N}}\left[2D_Y^2+\tsum_{k\in\B}\tfrac{D_X^2M_\Phi^2}{M_G^2k}\right] 
\leq \tfrac{2M_G }{D_X\sqrt{N}}\left[2D_Y^2+\tsum_{k=N/2}^N\tfrac{D_X^2M_\Phi^2}{M_G^2k}\right] \\
& \leq \tfrac{2M_G }{D_X\sqrt{N}}[2D_Y^2+\log 2D_X^2\tfrac{M_\Phi^2}{M_G^2}] \leq \tfrac{8M_\Phi D_Y}{\sqrt{N}}\max\{ \nu, \tfrac{1}{\nu}\}.
\end{aligned}
$$
Similarly, part b) follows from Theorem~\ref{Theorem 4}.b). 

\end{proof}

By Corollary~\eqref{Coro 4}, the CSPA method applied to \eqref{para1}-\eqref{para2} can achieve an ${\cal O}(1/\sqrt{N})$ rate of convergence.

\subsection{CSPA with strong convexity assumptions}
In this subsection, we modify problem \eqref{para1}-\eqref{para2} by imposing certain strong convexity assumptions to $\Phi$ and $G$ with respect to $y$ and $x$, respectively, i.e.,
$\exists \mu_\Phi,\mu_G >0 $, s.t.
\begin{align}
\Phi(x,y_1,\zeta)\geq \Phi(x,y_2,\zeta)+\langle \Phi'(x,y_2,\zeta),y_1-y_2\rangle+\tfrac{\mu_\Phi}{2}\|y_1-y_2\|^2,\ \ \forall y_1,y_2\in Y. \label{SCassumption3}\\
G(x_1,\xi)\geq G(x_2,\xi)+\langle G'(x_2,\xi),x_1-x_2\rangle+\tfrac{\mu_G}{2}\|x_1-x_2\|^2,\ \ \forall x_1,x_2\in X.\label{SCassumption4}
\end{align}
We also assume that the pair of solutions $(x^*,y^*)$ exists for problem \eqref{para1}-\eqref{para2}. Our main goal in this subsection is to estimate the convergence properties of the CSPA algorithm under these new assumptions.

We need to modify the probability distribution \eqref{PR} used in the CSPA algorithm as follows. Given the stepsize $\gamma_k$, modulus $\mu_G$ and $\mu_\Phi$, and growth parameter $Q$ (see \eqref{QuadraticCondition}), {\color{black} let us define
\begin{equation}
a_k := (\mu_\Phi \gamma_k)/Q
\mbox{ and } A_k := \left\{
            \begin{array}{ll}
              1, & \hbox{$k=1$;} \\
              \prod_{i\leq k,\ i\in\B}(1-a_i), & \hbox{$k>1$,}
            \end{array}
          \right.
\end{equation}
and denote
\begin{equation}
b_k := (\mu_G \gamma_k)/Q
\mbox{ and } B_k := \left\{
            \begin{array}{ll}
              1, & \hbox{$k=1$;} \\
              \prod_{i=1}^k(1-b_i), & \hbox{$k>1$.}
            \end{array}
          \right.
\end{equation}
Also the probability distribution of $R$ is modified to
\begin{equation}\label{PR2}
\Prob\{R=k\}=\tfrac{\gamma_k/A_k}{\tsum_{i\in \mathcal{B}}\gamma_i/A_i},k\in \mathcal{B}.
\end{equation}}
The following result shows some simple but important properties for the modified CSPA method applied to problem \eqref{para1}-\eqref{para2}.
\begin{prop}\label{Proposition 4}
For any $s\leq k\leq m$, we have
\begin{align}
    \tsum_{k\in \B}\tfrac{\gamma_k}{A_k} [\Phi(x_k,y_k,\zeta_k)-\Phi(x_k,y,\zeta_k)]\leq (1-\tfrac{\mu_\Phi \gamma_s}{Q})V_Y(y_s,y)+\tfrac{1}{2}\tsum_{k\in \mathcal{B}}\tfrac{\gamma_k^2}{A_k}\| \Phi'(x_k,y_k,\zeta_k)\|_*^2, \ \ \forall y\in Y\label{4.1}\\
\tsum_{k=\tau(s)}^{\tau(N)}\tfrac{\gamma_k}{B_k}\left[\eta_k-G(x,\xi_k)\right] \leq \left(1-\tfrac{\mu_G \gamma_s}{Q}\right)V_X(x_s,x) +\tfrac{1}{2}\tsum_{k=\tau(s)}^{\tau(N)}\tfrac{\gamma_k^2}{B_k}\| G'(x_k,\xi_k)\|_*^2,\ \ \forall x \in X .\label{4.2}
\end{align}
\end{prop}
\begin{proof}
Using Lemma~\ref{Lemma 1} and the strong convexity of $\Phi$ w.r.t. $y$, for $k\in B$, we have
$$
\begin{aligned}
    V_Y(y_{k+1},y) &\leq V_Y(y_k,y)-\gamma_k\langle \Phi'(x_k,y_k,\zeta_k),y_k-y\rangle+\tfrac{1}{2}\gamma_k^2\| \Phi'(x_k,y_k,\zeta_k)\|_*^2 \\
     &\leq V_Y(y_k,y)-\gamma_k\left[\Phi(x_k,y_k,\zeta_k)-\Phi(x_k,y,\zeta_k)+\tfrac{\mu_\Phi}{2}\|y_k-y\|^2\right]+\tfrac{1}{2}\gamma_k^2\| \Phi'(x_k,y_k,\zeta_k)\|_*^2 \\
     &\leq \left(1-\tfrac{\mu_\Phi \gamma_k}{Q}\right)V_Y(y_k,y) - \gamma_k[\Phi(x_k,y_k,\zeta_k)-\Phi(x_k,y,\zeta_k)] + \tfrac{1}{2}\gamma_k^2\| \Phi'(x_k,y_k,\zeta_k)\|_*^2.
\end{aligned}
$$
Also note that $V_Y(y_{k+1},y)=V_Y(y_k,y)$ for all $k\in\N$. Summing up these relations for all $k\in \B\cup \N$ and using Lemma \ref{Lemma 3}, we obtain
\begin{equation}\label{4.7}
    \tfrac{V_Y(y_{N},y)}{A_{N+1}} \leq \left(1-\tfrac{\mu_\Phi \gamma_s}{Q}\right)V_Y(y_s,y) - \tsum_{k\in \mathcal{B}}\tfrac{\gamma_k}{A_k}[\Phi(x_k,y_k,\zeta_k)-\Phi(x_k,y,\zeta_k)]+\tfrac{1}{2}\tsum_{k\in \mathcal{B}}\tfrac{\gamma_k^2}{A_k}\| \Phi'(x_k,y_k,\zeta_k)\|_*^2.
\end{equation}
Similarly for $\tau(s)\leq k\leq \tau(N)$, we have
$$
\begin{aligned}
    V_X(x_{k+1},x) &\leq V_X(x_k,x)-\gamma_k\langle G'(x_k,\xi_k),x_k-x\rangle+\tfrac{1}{2}\gamma_k^2\| G'(x_k,\xi_k)\|_*^2 \\
     &\leq V_X(x_k,x)-\gamma_k\left[G(x_k,\xi_k)-G(x,\xi_k)+\tfrac{\mu_G}{2}\|x_k-x\|^2\right]+\tfrac{1}{2}\gamma_k^2\| G'(x_k,\xi_k)\|_*^2 \\
     &\leq \left(1-\tfrac{\mu_G \gamma_k}{Q}\right)V_X(x_k,x) - \gamma_k[G(x_k,\xi_k)-G(x,\xi_k)] + \tfrac{1}{2}\gamma_k^2\| G'(x_k,\xi_k)\|_*^2,
\end{aligned}
$$
Summing up these relations for $\tau(s)\leq k\leq \tau(N)$ and using Lemma \ref{Lemma 3}, we have
\begin{equation}\label{4.8}
    \tfrac{V_X(x_{N+1},x)}{A_N} \leq \left(1-\tfrac{\mu_G \gamma_s}{Q}\right)V_X(x_s,x) - \tsum_{k=\tau(s)}^{\tau(N)}\tfrac{\gamma_k}{A_k}[\eta_k-G(x,\xi_k)]+\tfrac{1}{2}\tsum_{k=\tau(s)}^{\tau(N)}\tfrac{\gamma_k^2}{A_k}\| G'(x_k,\xi_k)\|_*^2.
\end{equation}
Using the facts that $V_Y(y_{N+1},y)/A_N\geq 0$ and $V_X(x_{N+1},x)/A_N \geq 0$, and rearranging the terms in \eqref{4.7} and \eqref{4.8}, we obtain \eqref{4.1} and \eqref{4.2}, respectively.

\end{proof}

\vgap

Lemma \ref{Lemma 7} below provides a sufficient condition which guarantees that the output solution $(\bar x_R,{y}_R)$ is well-defined.
\begin{lem}\label{Lemma 7}
{\color{black} The following statements hold.
\begin{description}
  \item[a)] Under Assumption~\ref{CSPA_assump1}, if for any $\lambda>0$, we have
\begin{equation}\label{4.3.1}
    \tfrac{N-s+1}{2}\min_{k\in \N}{\tfrac{\gamma_k\eta_k}{B_k}} > \left(1-\tfrac{\mu_G\gamma_s}{Q}\right)\lambda D_X^2+\lambda\tfrac{M_G^2}{2}\tsum_{k=s}^{N}\tfrac{\gamma_k^2}{B_k},
\end{equation}
then $\Prob\{|\mathcal{B}|\geq \tfrac{N-s+1}{2}\} \geq 1- 1/\lambda$.
  \item[b)] Under Assumption~\ref{CSPA_lighttail}, if for any $\lambda>0$, we have
\begin{equation}\label{4.3.2}
    \tfrac{N-s+1}{2}\min_{k\in \N}{\tfrac{\gamma_k\eta_k}{B_k}} > \left(1-\tfrac{\mu_G\gamma_s}{Q}\right) D_X^2+(1+\lambda)\tfrac{M_G^2}{2}\tsum_{k=\tau(s)}^{\tau(N)}\tfrac{\gamma_k^2}{B_k}
    +\lambda\sigma\sqrt{\tsum_{k=s}^{N}\tfrac{\gamma_k^2}{B_k^2}},
\end{equation}
then $\Prob\{|\mathcal{B}|\geq \tfrac{N-s+1}{2}\} \geq 1- \exp\{-\lambda\}-\exp\{-\lambda^2/3\}$.
\end{description}}
\end{lem}
\begin{proof}
The proof is similar to the one of Lemma~\ref{Lemma 6} and hence the details are skipped.
\end{proof}

\vgap

Now let us establish the rate of convergence of the modified CSPA method for problem \eqref{para1}-\eqref{para2}.
\begin{thm}\label{Theorem 5}
Suppose that $\{\gamma_k\}$ and $\{\eta_k\}$ are chosen according to Lemma~\ref{Lemma 7}. Then
  under Assumption~\ref{CSPA_assump1}, we have for any $\lambda>0$,
\begin{align}
 &\Prob\left\{\bbe_R[\phi(\bar x_R,y_R)-\phi(\bar x_R,y^*(\bar x_R))] \geq \lambda\left(\tsum_{k\in B}\tfrac{\gamma_k}{A_k}\right)^{-1}\left[(1-\tfrac{\mu_\Phi\gamma_s}{Q})D_Y^2+\tfrac{M_\Phi^2}{2}\tsum_{k\in B}\tfrac{\gamma_k^2}{A_k}\right]\right\}\leq \tfrac{1}{\lambda},\label{4.4.1}\\
 &\Prob\left\{g(\bar x_R) \geq  \eta_R+\lambda\sigma\tfrac{\sqrt{\tsum_{k=\tau(s)}^{\tau(N)}\gamma_k^2/B_k^2}}{\tsum_{k=\tau(s)}^{\tau(N)}\gamma_k/B_k}\right\}\leq \tfrac{1}{\lambda^2}.\label{4.4.2}
\end{align}
  In addition,
  under Assumption~\ref{CSPA_lighttail}, we have for any $\lambda>0$,
  \begin{align}
 &\Prob\left\{\bbe_R[\phi(\bar x_R,y_R)-\phi(\bar x_R,y^*(\bar x_R))] \geq K_0+\lambda K_1\right\}\leq \exp\{-\lambda\}+\exp\{-\lambda^2/3\},\label{4.5.1}\\
 &\Prob\left\{g(\bar x_R) \geq \eta_R+\lambda\sigma\tfrac{\sqrt{\tsum_{k=\tau(s)}^{\tau(N)}\gamma_k^2/B_k^2}}{\tsum_{k=\tau(s)}^{\tau(N)}\gamma_k/B_k}\right\}\leq \exp\{-\lambda^2/3\},\label{4.5.2}
\end{align}
where $K_0 = \left(\tsum_{k\in \B}\tfrac{\gamma_k}{A_k}\right)^{-1}\left[(1-\tfrac{\mu_\Phi\gamma_s}{Q})D_Y^2+\tfrac{M_\Phi^2}{2}\tsum_{k\in \B}\tfrac{\gamma_k^2}{A_k}\right]$ \newline
and $K_1 =  \left(\tsum_{k\in \B}\tfrac{\gamma_k}{A_k}\right)^{-1}\left[M_\Phi^2\tsum_{k\in \B}\tfrac{\gamma_k^2}{A_k}+4M_\Phi D_Y\sqrt{\tsum_{k\in \B}\tfrac{\gamma_k^2}{A_k^2}}\right]$.
\end{thm}
\begin{proof}
The proof is similar to the proof of Theorem~\ref{Theorem 4}, and hence the details are skipped.
\end{proof}

\vgap

Now we provide a specific selection of $\{\gamma_k\}$ and $\{\eta_k\}$ that satisfies the condition of Lemma~\ref{Lemma 7}. While the selection of $\eta_k$ only depends on iteration index $k$, i.e.,  for some $\rho \in (0,1)$,
\begin{equation}\label{eta4}
\eta_k = \tfrac{QM_G^2}{\tau(k)\mu_G\rho}
\end{equation}
under Assumption~\ref{CSPA_assump1} and
\begin{equation}\label{eta5}
\eta_k = \tfrac{QM_G^2}{\tau(k)\mu_G}(1+\tfrac{1}{\rho})+\tfrac{\sigma}{\tau(k)\rho}, 
\end{equation}
under Assumption~\ref{CSPA_lighttail},
the selection of $\gamma_k$ depends on the particular position of iteration index $k$ in set $\B$ or $\N$. More specifically, let $\tau_{\B(k)}$ and $\tau(k)$ be the position of index $k$ in set $\B$ and set $\N$, respectively (for example, $\B= \{1,3,5,9,10\}$ and $\N = \{2,4,6,7,8\}$. If $k=9$, then $\tau_{B(k)}=4$). We define $\gamma_k$ as
\begin{equation}\label{gamma4}
\gamma_k = \left\{
                  \begin{array}{ll}
                    \tfrac{2Q}{\mu_\Phi(\tau_{\B(k)}+1)}, & \hbox{$k\in \B$;} \\
                    \tfrac{2Q}{\mu_G(\tau(k)+1)}, & \hbox{$k\in \N$.}
                  \end{array}
                \right.
\end{equation}
Such a selection of $\gamma_k$ can be conveniently implemented by using two separate counters in each iteration to represent $\tau_{\B(k)}$ and $\tau(k)$.

\begin{cor}\label{Coro 5}
Let $s=\tfrac{N}{2}+1$, $\eta_k$ and $\gamma_k$ be given in \eqref{eta4}, \eqref{eta5} and \eqref{gamma4}, respectively. Then we have under Assumption~\ref{CSPA_assump1}, we have for any $\lambda>0$,
\begin{align*}
 &\Prob\left\{\bbe_R[\phi(\bar x_R,y_R)-\phi(\bar x_R,y^*(\bar x_R))] \leq \lambda\tfrac{8QM_\Phi^2}{(N+2)\mu_\Phi}\right\}\geq (1-\tfrac{1}{\lambda})(1-\rho),\\
 &\Prob\left\{g(\bar x_R) \leq \lambda\tfrac{2QM_G^2}{\rho N\mu_G}\right\}\geq 1-\tfrac{1}{\lambda^2}.
\end{align*}
In addition, under Assumption~\ref{CSPA_lighttail}, we have for any $\lambda>0$,
  \begin{align*}
 &\Prob\left\{\bbe_R[\phi(\bar x_R,y_R)-\phi(\bar x_R,y^*(\bar x_R))] \leq K_0+\lambda K_1\right\}\\
 &\geq (1-\exp\{-\lambda\}-\exp\{-\lambda^2/3\})(1-\exp\{-\rho\}-\exp\{-\rho^2/3\}),\\
 &\Prob\left\{g(\bar x_R) \leq\tfrac{2QM_G^2}{\rho N\mu_G}+\lambda\tfrac{2\sigma}{\sqrt{N}}\right\}\geq 1-\exp\{-\lambda^2/3\},
\end{align*}
where $K_0 = 8QM_\Phi^2/[(N+2)\mu_\Phi]$ and $K_1 =8QM_\Phi^2/[(N+2)\mu_\Phi]+ 64M_\Phi D_Y/\sqrt{N}$.
\end{cor}
\begin{proof}
The proof is similar to the proof of Corollary~\ref{Coro 4} and hence the details are skipped.
\end{proof}

Note that Corollary~\ref{Coro 5}.a) implies an ${\cal O}(1/N)$ rate of convergence, while Corollary~\ref{Coro 5}.b) 
show an ${\cal O}(1/\sqrt{N})$ rate of convergence with much improved dependence on $\lambda$. One possible approach
to improve the result in part b) is to shrink the feasible set $Y$ from time to time in order to obtain an ${\cal O} (1/N)$
rate of convergence (see \cite{GhaLan13-1}).

\setcounter{equation}{0}
\section{Numerical Experiment}
 In this section, we present some numerical results of our computational experiments for solving two problems: an asset allocation problem with conditional value at risk (CVaR) constraint and a parameterized classification problem. More specifically, we report the numerical results obtained from the CSA and CSPA method applied to these two problems in Subsection 4.1 and 4.2, respectively.

 \subsection{Asset allocation problem}
 Our goal of this subsection is to examine the performance of the CSA method applied to the CVaR constrained problem in \eqref{CVaR}.

Apparently, there is one problem associated with applying the CSA algorithm to this model -- the feasible region $X$ is unbounded. Lan, Nemirovski and Shapiro (see \cite{lns11} Section 4.2) show that $\tau$ can be restricted to
$
\left[\underline{\mu}+\sqrt{\tfrac{\beta}{1-\beta}}\sigma, \bar{\mu}+\sqrt{\tfrac{1-\beta}{\beta}}\sigma\right],
$
where $\underline{\mu}:= \min_{y\in Y}\{-\bar{\xi}^T y\}$ and $\bar{\mu}:= \max_{y\in Y}\{-\bar{\xi}^T y\}$.

In this experiment, we consider four instances. The first three instances are randomly generated according to the factor model in Goldfarb and Iyengar (see Section 7 of \cite{goldfarb2003robust} ) with different number of stocks ({\color{black} $d =500$, $1000$ and $2000$}), while the last instance consists of the 95 stocks from $S\&P100$ (excluding SBC, ATI, GS, LU and VIA-B) obtained from \cite{wang2008sample}, the mean $\bar{\xi}$ and covariance $\Sigma$ are estimated by the historical monthly data from 1996 to 2002. The reliability level $\beta = 0.05$, {\color{black} the number of samples to estimate $g(x)$ is $J = 100$} and the number of samples used to evaluate the solution is $n= 50,000$.
{\color{black} It is worth noting that, by utilizing the linear structure of $\xi^Tx$ (where $x\in \bbr^d $) in constraint function, in k-th iteration we generate J-sized i.i.d. samples of $\bar \xi:=\xi^Tx_k$ (with dimension 1)
to estimate $\xi^Tx$ in constraint function, instead of J-sized i.i.d. samples of $\xi$ (with dimension $d$).}
For SAA algorithm, the deterministic SAA problem to \eqref{CVaR} is defined by
\begin{equation} \label{CVaR_SAA}
\begin{array}{ll}
  \min_{x, \tau} &  -\mu^T x  \\
  \mbox{s.t.} & \tau + \tfrac{1}{\beta N}\tsum_{i=1}^N[-\xi_i^T x-\tau]_+\leq 0, \\
   &  \tsum_{i=1}^n x_i=1, x \ge 0,
\end{array}
\end{equation}
We implemented the SAA approach by using Polyak's subgradient method for solving convex programming problems with function constraints (see \cite{polyak1967general}). The main reasons why we did not use the linear programming (LP) method to \eqref{CVaR_SAA} include: 1) problem \eqref{CVaR_SAA} might be infeasible for some instances; and 2) we tried the LP method with CVX toolbox for an instance with 500 stocks and the CPU time is thousands times larger than that of the CSA method.
In our experiment, we adjust the stepsize strategy by multiplying $\gamma_k$ and $\eta_k$ with some scaling parameters $c_g$ and $c_e$, respectively. These parameters are chosen as a result of pilot runs of our algorithm (see \cite{lns11} for more details). We have found that the ``best parameters" in Table~\ref{CSA_expara} slightly outperforms other parameter settings we have considered.
\begin{table}[h]
\begin{center}
\caption{The stepsize factor}\label{CSA_expara}
\footnotesize
\begin{tabular}{|c|c|c|c|}
  \hline      &   & best $c_g$ & best $c_e$ \\
  \hline    Number  &500&  0.5 & 0.005\\
        of stocks &1000& 0.5 & 0.05  \\
           &2000& 0.5 & 0.05 \\
  \hline
\end{tabular}
\end{center}
\end{table}
\\

{\bf Notations in Tables~\ref{CSA_exd500}-\ref{CSA_exCVAR}.}
\begin{description}
  \item [N:] the sample size( the number of steps in SA, and the size of the sample used to SAA approximation).
  \item [Obj.:] the objective function value of our solution, i.e. the loss of the portfolio.
  \item [Cons.:] the constraint function value of our solution.
  \item [CPU:] the processing time in seconds for each method.
\end{description}

\begin{table}[h]
\begin{center}
\caption{Random Sample with 500 Assets}\label{CSA_exd500}
\footnotesize
\begin{tabular}{|c|c|c|c|c|c|}
  \hline      &   & $N$=500 & $N$=1000 & $N$=2000 & $N$=5000 \\
  \hline      &Obj.& -4.883 & -4.870 & -4.953 & -4.984\\
         CSA &Cons.& 5.330 & 4.096 & 5.167 & 2.859 \\
              &CPU& 1.671e-01 & 3.383e-01 & 6.271e-01 &1.470e+00 \\
  \hline      &Obj.& -4.978 & -4.981 & -4.977 & -4.977 \\
         SAA  &Cons.& 4.372 & 3.071 & 2.330 & 2.249  \\
              &CPU& 2.031e+00 & 9.926e+00 & 4.132e+01 & 2.591e+02 \\
  \hline
\end{tabular}
\end{center}
%
\begin{center}
\caption{Random Sample with 1000 Assets}\label{CSA_exd1000}
\footnotesize
\begin{tabular}{|c|c|c|c|c|c|}
  \hline      &   & $N$=500 & $N$=1000 & $N$=2000 & $N$=5000 \\
  \hline      &Obj.& -4.532 & -4.704 &-4.838 & -4.949\\
         CSA &Cons.& 27.660 & 24.901 & 23.825 & 20.785 \\
              &CPU& 4.193e-01 &  8.578e-01 & 1.659e+00 & 4.001e+00 \\
  \hline      &Obj.& -4.965 & -4.981 & -4.981 & -4.977 \\
         SAA  &Cons.& 60.421 &  47.745 & 33.940 & 20.357  \\
              &CPU& 1.513e+01 & 5.954e+01 & 2.774e+02 & 1.524e+03 \\
  \hline
\end{tabular}
\end{center}

\begin{center}
\caption{Random Sample with 2000 Assets}\label{CSA_exd2000}
\footnotesize
\begin{tabular}{|c|c|c|c|c|c|}
  \hline      &   & $N$=500 & $N$=1000 & $N$=2000 & $N$=5000 \\
  \hline      &Obj.& -4.299 & -4.077 & -4.355 & -4.859\\
         CSA &Cons.& 144.92& 112.54 & 89.74 & 82.65  \\
              &CPU& 1.374e+00 & 2.810e+00 & 5.538e+00 & 2.716e+01 \\
  \hline      &Obj.& -4.752 & -4.699 & -4.721  & -4.727 \\
         SAA  &Cons.& 279.43 & 218.96 & 147.93 & 94.46  \\
              &CPU& 1.968e+01 & 6.571e+01 & 2.940e+02 & 3.697e+03 \\
  \hline
\end{tabular}
\end{center}

\end{table}

\begin{table}[h]
\begin{center}
\caption{Comparing the CSA and SAA for the CVaR model}\label{CSA_exCVAR}
\footnotesize
\begin{tabular}{|c|c|c|c|c|c|c|}
  \hline      &   & N=500 & N=1000 & N=2000 & N=5000 & N=10000 \\
  \hline      &Obj.& -3.531 & -3.537 & -3.542 & -3.548 & -3.560\\
         CSA &Cons.& 3.382e+00 & 2.188e-01 & 1.106e-01 & 2.724e-01  & -7.102e-01  \\
              &CPU& 8.315e-02 & 1.422e-01 & 2.778e-01 & 7.251e-01 & 1.415e+00  \\
  \hline      &Obj.& -3.530 & -3.541 & -3.541 & -3.544 & -3.559\\
         SAA  &Cons.& 3.385e+00 & 7.163e-01 & 6.989e-01 & 6.988e-01 &  7.061e-01\\
              &CPU& 3.155e+00 & 1.221e+01 & 4.834e+01 & 3.799e+02 & 1.462e+03\\
  \hline
\end{tabular}
\end{center}

\end{table}
The following conclusions can be made from the numerical results. First, as far as the quality of solutions is concerned, the CSA method is at least as good as SAA method and it may outperform SAA for some instances especially as $N$ increases. Second, the CSA method can significantly reduce the processing time than SAA method for all the instances.
\subsection{Classification and metric learning problem} \label{sec_num_metric}
In this subsection, our goal is to examine the efficiency of the CSPA algorithm applied to a classification problem with the metric as parameter. In this experiment, we use the expectation of hinge loss function, described in \cite{Singer07pegasos:primal}, as objective function, and formulate the constraint with the loss function of metric learning problem in \cite{DuBaMaWa11}, see formal definition in \eqref{svm}-\eqref{metric}. For each $i,j$, we are given samples $u_i, u_j\in \bbr^d$ and a measure $b_{ij} \geq 0$ of the similarity between the samples $u_i$ and $u_j$ ($b_{ij} =0$ means $u_i$ and $u_j$ are the same). The goal is to learn a metric $A$ such that $\langle (u_i-u_j), A(u_i-u_j) \rangle\approx b_{ij}$, and to do classification among all the samples $u$ projected by the learned metric $A$.

For solving this class of problems in machine learning, one widely accepted approach is to learn the metric in the first step and then solve the classification problem with the obtained optimal metric. However, this approach is not applicable to the online setting since once the dataset is updated with new samples, this approach has to go through all the samples to update $A$ and $\omega$. On the other hand, the CSPA algorithm optimizes the metric $A$ and classifier $\omega$ simultaneously, and only needs to take one new sample in each iteration.

In this experiment, our goal is to test the solution quality of the CSPA algorithm with respect to the number of iterations. More specifically, we consider $2$ instances of this problem with different dimension ($d = 100$ and $200$, respectively). Since we are dealing with the online setting, our sample size for training $A$ and $\omega$ is increasing with the number of iterations. The size for the sample
used to estimate the parameters and the one used to evaluate the quality of solution (or testing sample) are set to $100$ and $10,000$, respectively. Within each trial, we test the objective and constraint value of the output solution over training sample and testing sample, respectively.
Since $R$ is randomly picked up from all the series $\{\bar x_k, y_k\}$, we generate 5 candidate $R$, instead of one, in order to increase the probability of getting a better solution. Intuitively, the latter solutions in the series should be better than the earlier ones, hence, we also put the last pair of the solution $(\bar x_N, y_N)$ into the candidate list. In each trial, we compare these 6 candidate solutions. First, we choose three pairs with smallest constraint function values, then, choose the one with the smallest objective function value from these three selected solutions as our output solution.


Table~\ref{CSPA_exd100} and Table~\ref{CSPA_exd200} shows the CSPA method decreases the objective value and constraint value as the sample size (number of iterations $N$) increases.
These experiments demonstrate that we can improve both the metric and the classifier simultaneously by using the CSPA method as more and more data are collected.\\
{\bf Notations in Table~\ref{CSPA_exd100} and~\ref{CSPA_exd200}.}\\
Obj. Train: The objective function value using training sample at the output solution.\\
Cons. Train: The constraint function value using training sample at the output solution.\\
Obj. Test: The objective function value using testing sample at the output solution.\\
Cons. Test: The constraint function value using testing sample at the output solution.

\begin{table}[h]
\begin{center}
\caption{d = 100}\label{CSPA_exd100}
\footnotesize
\begin{tabular}{|c|c|c|c|c|}
  \hline  $N$& Obj. Train & Cons. Train & Obj. Test & Cons. Test\\
  \hline   100& 3.175 & 3.056  & 1.042 & 3.068 \\
  \hline   200 & 2.737 & 3.058 & 0.811 & 3.006 \\
  \hline   600 & 0.654 & 3.077 & 0.157 & 3.104 \\
  \hline   800 & 0.529 & 3.087  &  0.126 & 3.102 \\
  \hline   1000 & 0.398 & 3.057 & 0.102 & 3.082 \\
  \hline
\end{tabular}
\end{center}

\begin{center}
\caption{d = 200}\label{CSPA_exd200}
\footnotesize
\begin{tabular}{|c|c|c|c|c|}
  \hline  $N$& Obj. Train & Cons. Train & Obj. Test & Cons. Test\\
  \hline   100& 0.716 & 1.137  & 0.699 & 1.132 \\
  \hline   200 & 0.374 & 1.061 & 0.371 & 1.030\\
  \hline   1000 & 0.360 & 1.020 & 0.364 & 1.031 \\
  \hline   2000 & 0.351 & 1.016  &  0.355 & 1.030 \\
  \hline   5000 & 0.291 & 0.951 & 0.135 & 0.989 \\
  \hline
\end{tabular}
\end{center}
\end{table}

\section{Conclusions}
In this paper, we present a new stochastic approximation type method, the CSA method, for solving the stochastic convex optimization problems with function or expectation constraints. 
Moreover, we show that a variant of CSA method, the CSPA method, is applicable to a class of parameterized stochastic problem in \eqref{para1}-\eqref{para2}. We show that these methods exhibit theoretically optimal rate of convergence for solving a few different classes of function or expectation constrained stochastic optimization problems and demonstrated their effectiveness through some preliminary numerical experiments.

\renewcommand\refname{Reference}

\bibliographystyle{abbrv}
\bibliography{glan-bib}

\end{document}